\documentclass[11pt]{amsart}
\usepackage[a4paper,top=1.25in, bottom=1.25in, left=1.1in, right=1.1in]{geometry}
\usepackage{parskip}

\usepackage[T1]{fontenc}
\usepackage{baskervald}

\usepackage{amsmath,amsfonts,amssymb,graphicx,csquotes,mathtools}
\usepackage{mathrsfs}
\usepackage[dvipsnames]{xcolor} 

\usepackage{tikz-cd}
\usepackage{tikz}
\usepackage{tikz-3dplot} 
\usepackage{subcaption}

\usepackage[backend=biber, style=alphabetic, sorting=nyt, maxbibnames=99]{biblatex}
\renewbibmacro{in:}{%
  \ifentrytype{article}{}{\printtext{\bibstring{in}\intitlepunct}}}
\usepackage{hyperref}
\hypersetup{
  colorlinks=true,
  linkcolor=MidnightBlue,
  citecolor=OliveGreen,
urlcolor=Emerald,
}
\usepackage{color}
\usepackage[dvipsnames]{xcolor}

\usepackage{nameref}
\makeatletter
\let\orgdescriptionlabel\descriptionlabel
\renewcommand*{\descriptionlabel}[1]{%
  \let\orglabel\label
  \let\label\@gobble
  \phantomsection
  \edef\@currentlabel{#1}%
  \let\label\orglabel
  \orgdescriptionlabel{#1}%
}

\usepackage{cleveref}

\usepackage{enumerate}
\renewenvironment{description}[0]
{\list{}{\labelwidth=0pt \leftmargin=0.7cm
    }}
{\endlist}

\numberwithin{equation}{section}

\usepackage{amsthm}

\theoremstyle{definition}
\newtheorem{definition}{Definition}[section]

\theoremstyle{plain}

\newtheorem{lemma}[definition]{Lemma}
\newtheorem{sublemma}[definition]{Sublemma}

\newtheorem{theorem}[definition]{Theorem}
\newtheorem*{theorem*}{Theorem}

\newtheorem{manualtheoreminner}{Theorem}
\newenvironment{manualtheorem}[1]{%
  \renewcommand\themanualtheoreminner{#1}%
  \manualtheoreminner
}{\endmanualtheoreminner}

\newtheorem{proposition}[definition]{Proposition}

\theoremstyle{plain}
\newtheorem{corollary}[definition]{Corollary}

\theoremstyle{remark}
\newtheorem{remark}[definition]{Remark}

\theoremstyle{definition}

\theoremstyle{plain}

\makeatother

\usepackage{exercises}

\newcommand{\Leb}{\operatorname{Leb}}

\newcommand{\Lip}{\operatorname{Lip}}
\newcommand{\Int}{\operatorname{Int}}
\newcommand{\Card}{\operatorname{Card}}

\newcommand{\eps}{\varepsilon}

\newcommand{\quand}{\quad\text{ and } \quad}

\newcommand{\cB}{{\mathcal{B}}}
\newcommand{\cE}{{\mathcal{E}}}
\newcommand{\cF}{{\mathcal{F}}}

\newcommand{\cL}{{\mathcal{L}}}
\newcommand{\cM}{{\mathcal{M}}}
\newcommand{\cP}{{\mathcal{P}}}
\newcommand{\cS}{{\mathcal{S}}}
\newcommand{\cT}{{\mathcal{T}}}

\newcommand{\hZ}{{\widehat{Z}}}

\newcommand{\R}{{\mathbb R}}

\newcommand{\Z}{{\mathbb Z}}

\newcommand{\N}{{\mathbb N}}
\newcommand{\E}{{\mathbb E}}
\newcommand{\PP}{{\mathbb P}}

\hypersetup{
  pdfauthor={Douglas Coates, Ian Melbourne and Amin Talebi},
  pdftitle={Natural measures and statistical properties of non-statistical maps with multiple neutral fixed points},
  pdfkeywords={},
  pdfsubject={},
  pdflang={English}
}

\begin{document}

\title[Natural measures and statistical properties of non-statistical maps]{Natural measures and statistical properties of \\ non-statistical maps with multiple neutral fixed points}

\author{Douglas Coates}
\address{Instituto de Matemática, Universidade Federal do Rio Janeiro, Brazil}
\email{coates@im.ufrj.br }
\urladdr{https://douglascoates.github.io}

\author{Ian Melbourne}
\address{Mathematics Institute, University of Warwick, Coventry CV4 7AL, United Kingdom}
\email{i.melbourne@warwick.ac.uk}
\urladdr{https://homepages.warwick.ac.uk/~maslaq/research}

\author{Amin Talebi}
\address{\parbox{\linewidth}{Department of Mathematical Sciences, Sharif University of Technology, Iran,\\
Institute for Research in Fundamental Sciences (IPM), Iran}}
\email{amin.talebi@sharif.edu}

 \thanks{Douglas Coates was partially supported by the São Paulo Research Foundation (FAPESP) grant 2022/16259-2 and the Rio de Janeiro Research Foundation (FAPERJ) grant E-26/200.027/2025.
 }
\thanks{Ian Melbourne was partially supported by the São Paulo Research Foundation (FAPESP) grant 2024/22093-5.
}
 \thanks{ Amin Talebi is partially supported by INSF grant no.\ 4001845. \\}
\date{\today}

\begin{abstract}
  We show that a large class of infinite measure preserving dynamical systems that do not admit physical measures nevertheless exhibit strong statistical properties. In particular, we give sufficient conditions for existence of a distinguished natural measure~\( \nu \) such that the pushforwards of any absolutely continuous probability measure converge to~\( \nu \). Moreover, 
  we obtain a distributional limit law for empirical measures.
  Both of these are new results for 
  intermittent maps with at least two neutral fixed points preserving an infinite $\sigma$-finite absolutely continuous measure.

  We also extend existing results on the characterisation of the set of almost sure limit points for empirical measures. This result is new when there are at least three neutral fixed points.
\end{abstract}

\maketitle

\section{Introduction}

Let $f:X\to X$ be a dynamical system and let $m$ be a reference probability measure with respect to which we wish to study the statistical behaviour of almost every point. A basic object encoding the long-term behaviour of an orbit is the sequence of \emph{empirical measures}
\begin{equation}\label{eq:en}
e_n(x)\coloneqq \frac{1}{n}\sum_{j=0}^{n-1}\delta_{f^j x}
= \frac{1}{n}\sum_{j=0}^{n-1} f^j_*\delta_x .
\end{equation}
Understanding the asymptotic behaviour of $e_n(x)$ for typical points is one of the central problems in ergodic theory.

If $m$ is invariant under $f$, then Birkhoff’s ergodic theorem guarantees that $e_n(x)$ converges for $m$-almost every $x$.\footnote{Here and throughout, $X$ is assumed to be a topological space, and convergence of sequences of measures is with respect to the weak-$*$ topology unless otherwise specified.} If, however, $m$ is not invariant, then very different scenarios may occur. One common situation is that the system admits finitely many invariant probability measures which describe the statistical behaviour of $m$-almost every point. In this case, the long-term dynamics is governed by \emph{physical measures}: an invariant probability measure $\omega_0$ is called physical if there exists $E\subset X$ with $m(E)>0$ such that $e_n(x)\to \omega_0$ for all $x\in E$. This framework encompasses a large class of dynamical systems studied in the literature (see \cite{Alves-Bonatti-Viana2000,Ander-Vas2018,Bowen75,Bowen-Ruelle-75,Dol2000,Gan-Li-Viana-Yang-PVE-2021,yang_2021}).

More recently, there has been significant interest in \emph{nonstatistical maps}\footnote{Nonstatistical behaviour is also referred to as historical behaviour in the literature }, namely systems for which $e_n(x)$ fails to converge for $m$-almost every $x$; see~\cite{AG2022,pablo23,BC25,MB06,CroYanZha2020,Kriki&Soma,talebi2025} and references therein. These systems lie at the opposite extreme from those described above and, in particular, admit no physical measures. Their statistical description therefore requires new approaches.

For nonstatistical maps there are three directions that we explore in this paper:
\begin{itemize}
\item[(A)] Determination of the set of a.s.\ limit points of $e_n$.
\item[(B)] Characterisation of the distributional limits of $e_n$.
\item[(C)] Weak-$*$ convergence of 
$\frac{1}{n}\sum_{j=0}^{n-1} f^j_* \lambda$ where $\lambda\ll m$.
\end{itemize}
We have not seen direction~(B) addressed before. Also, we obtain to our knowledge the first results for direction~(A) in situations where the set of a.s.\ limit points is neither a singleton nor an interval. With regard to direction~(C), a probability measure $\nu$ is called \emph{natural}~\cite{BlaBun2003,CatEnr2012,JarTol2005,Mis2005} if there exists an absolutely continuous probability measure $\lambda$ such that
\[
\lim_{n\to\infty}\frac{1}{n}\sum_{j=0}^{n-1}f^j_*\lambda=\nu .
\]
(By the dominated convergence theorem, physical measures are natural.)

In this article, we show that a large class of infinite measure preserving dynamical systems that are nonstatistical, and hence do not admit physical measures nevertheless exhibit strong statistical properties along the lines of directions~(A)--(C) above.

Our results apply in particular to \emph{Pomeau-Manneville maps}, which were introduced in the 1980s by~\cite{PomMan1980} as models for intermittent turbulent or chaotic bursts interspersed with laminar behaviour. Subsequently, they have provided prototypical examples of nonuniformly hyperbolic dynamics in smooth ergodic theory and are regarded as a test-bed for new theoretical approaches in the area. The most studied examples are the \emph{Liverani-Saussol-Vaienti maps}~\cite{LivSauVai1999}, which form a special case of the intermittent maps analysed in detail by Thaler~\cite{Tha1980,Tha1983}.

In this paper, we focus on the maps considered by Thaler~\cite{Tha1980} as well as more recent examples of Coates \emph{et al.}~\cite{CoaLuzMuh2023}. Specifically, we study maps that preserve an infinite $\sigma$-finite absolutely continuous measure and have several ``equally sticky'' neutral fixed points.

For the sake of clarity, we first introduce our results for Thaler maps with two neutral fixed points~\cite{Tha2002} before discussing more general situations. 
Let $f : [0,1] \to [0,1]$ be an interval map with the following properties for some $\alpha\in(0,1]$, $b_1$, $b_2 > 0$, $c\in(0,1)$:
\begin{enumerate}
  \item\label{itm:tha1}
  \( 0 , 1 \) are neutral fixed points: \( f ' (0 ) = f' (1) = 1 \);
  \item
  $f|_{(0,c)}$ and $f|_{(c,1)}$ are $C^2$ diffeomorphisms onto $(0,1)$ admitting $C^2$ extensions onto $[0,1]$;

  \item
  $f'(x) > 1$ for all $x\in (0,1)$ and
  \( f \) is convex (resp.\ concave) on a neighbourhood of \( 0 \) (resp.\ \( 1 \));

  \item\label{itm:tha4}
  $fx - x \sim b_1 x^{1 + 1/\alpha}$ as $x \to 0$ and $fx - x \sim  b_2 (1-x) ^{ 1 + 1/\alpha}$ as $x \to 1$.
\end{enumerate}
By~\cite[Eq.~(1.5)]{Tha1983}, $f$ is conservative and ergodic with a unique (up to scaling)  invariant $\sigma$-finite absolutely continuous measure~$\mu$. Moreover, by~\cite[p.~71]{Tha1983}, $\mu([0,1]) = \infty$. However, it was shown in \cite[Theorem~2 and discussion thereafter]{AarThaZwe2005} that for almost every $x\in[0,1]$, the sequence of empirical measures $e_n(x)$
diverges for almost every $x\in[0,1]$ and hence $f$ is non-statistical.
Moreover, the set of weak-$*$ limit points is almost surely equal to the set 
\[
  \cS \coloneqq \{ \nu_p \coloneqq p \delta_{0} + (1 - p )\delta_1: p\in [0,1] \}
\]
of all convex combinations of the Dirac masses at the neutral fixed points. 

In the absence of physical measures, 
one may ask whether convergence holds when the point mass $\delta_x$ in~\eqref{eq:en} is replaced by other types of probability measure $\lambda$ and 
whether there is a distinguished limit measure $\nu_{\bar p}\in \cS$.

Our first result (corresponding to direction~C above) shows that there exists a unique natural measure $\nu_{\bar p}\in\cS$. 
Moreover, $\nu_{\bar p}$ attracts all absolutely continuous probability measures $\lambda$ and it is not necessary to take Ces\'aro averages.
(See~\cite{Kel2004} for similar results in the simpler symmetric setting with $b_1=b_2$, where $\bar p=\frac12$. See also~\cite{Zwe2002} for results on nonexistence of natural measures when $f$ is badly behaved near the neutral fixed points.)

\begin{theorem}[Existence and uniqueness of natural measures]
  \label{thm:tha-mixing}
  There exists a $\bar p\in(0,1)$, with corresponding measure $\nu_{\bar{p}} \in \cS$, 
  such that 
  $$
  \lim_{n\to\infty} f_{*}^n \lambda =  \nu_{ \bar{p} }
  \quad\text{for every absolutely continuous probability measure $\lambda$.}
  $$
\end{theorem}

\begin{remark} A formula for $\bar p$ is given in Section~\ref{sec:thaler}, namely $\bar p=c_1/(c_1+c_2)$ where $c_1,c_2$ are as in~\eqref{eq:tha-2}.
\end{remark}

An almost immediate consequence of Theorem~\ref{thm:tha-mixing}
is the following decay of correlations type result. 

\begin{corollary}\label{cor:tha-mixing}
  Suppose that $\varphi:[0,1]\to\R$ is bounded, measurable, and
  continuous at $0$ and $1$.
  Then
  $$
  \lim_{n\to\infty} \int \psi \cdot \varphi \circ f^n \, d\Leb = \int \psi \, d\Leb \int \varphi \, d \nu_{\bar{p}}\quad
  \text{for all $\psi \in L^1 ([0,1],\Leb)$. }
  $$
  Equivalently, 
  $
  \lim_{n\to\infty} \int \psi \cdot \varphi \circ f^n \, d \mu = \int \psi \, d \mu \int \varphi \, d \nu_{\bar{p}}$
  for all $\psi \in L^1 ([0,1], \mu )$.
\end{corollary}

We now return to consideration of the sequence $e_n$ of empirical measures.
Since $e_n$ fails to converge pointwise, it is natural to consider alternative modes of convergence such as
distributional convergence (direction~B above).

To describe the distributional convergence of $e_n$, we recall the classical arcsine law for occupation times of~\cite{Lev1939} as adapted by~\cite{Lam1958} and then by~\cite{Tha2002} to the current context.
Define the sequence of occupation times
$$
S_n \coloneqq \sum_{ j = 0 }^{ n - 1 } 1_{ B} \circ f^j,
$$
where $B\subset [0,1)$ is a closed interval containing $0$.
Also, we define the $[0,1]$-valued random variable $Z_{\alpha, p}$ for $\alpha\in(0,1]$, $p\in[0,1]$, as follows:
When $\alpha=1$, we set $Z_{1,p}\equiv p$. When $\alpha\in(0,1)$, let $Z_{\alpha, p}$ be the random variable with continuous
density 
\begin{equation}
  \label{eq:arcsine-density}
  \frac{ \hat{ p } \sin \alpha \pi }{ \pi } \frac{ t^{ \alpha }(1 - t)^{ \alpha - 1 }+t^{ \alpha - 1  }(1 - t)^{ \alpha } }{ \hat{ p }^{ 2 }t^{ 2 \alpha  } + 2\hat{ p } t^{ \alpha }(1 - t)^{  \alpha  } \cos \alpha \pi + (1 - t)^{ 2  \alpha } },
\end{equation}
where $\hat p= p^{-1}-1$.
Thaler~\cite[Theorem, p.~1291]{Tha2002} proved a distributional limit law for occupation times:
\[
  \tfrac1n S_n \text{ converges strongly in distribution to }  Z_{\alpha,\bar p}.
\]
When $\alpha\in(0,1)$, this 
means that $\lambda(\frac1n S_n\in I)\to \PP(Z_{\alpha,\bar p}\in I)$ for all absolutely continuous probability measures $\lambda$ on $[0,1]$ and all open intervals $I\subset[0,1]$.\footnote{We use standard probabilistic conventions, so on the left-hand side $\lambda(\frac1n S_n\in I)=\lambda(x\in [0,1]:\frac1n S_n(x)\in I)$ and on the right-hand side
$\PP(Z_{\alpha,\bar p}\in I)$ is the integral of the density~\eqref{eq:arcsine-density} over $I$.}
(When $\alpha=1$, we exclude intervals $I$ with an endpoint at $\bar p$.)
Equivalently, \( \frac1n S_{n*} \lambda \to \omega_0 \) 
where $\omega_0$ is the distribution $\omega_0(I)=\PP(Z_{\alpha,\bar p}\in I)$ of $Z_{\alpha,\bar p}$.

To each $[0,1]$-valued random variable $Z$, we can associate the $\cS$-valued random variable $\nu_Z=Z\delta_0+(1-Z)\delta_1$.

\begin{theorem}[Distributional limit law for empirical measures]
  \label{thm:tha-dist-conv}
  \[
    e_n \text{ converges strongly in distribution to }  \nu_{Z_{\alpha,\bar p}} .
  \]
\end{theorem}

Theorem~\ref{thm:tha-dist-conv}
describes how the $e_n(x)$ are asymptotically distributed on $\cS$, see Figure \ref{fig:D1}.

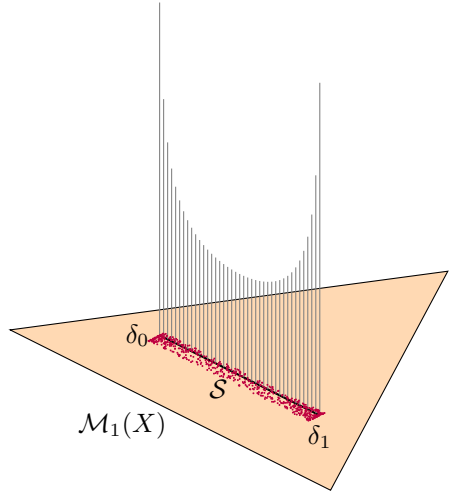
\begin{figure}
  \centering
  
  \tdplotsetmaincoords{60}{45} 

  \begin{tikzpicture}[tdplot_main_coords,scale=3]
    \coordinate (A) at (0,0,0);
    \coordinate (B) at (2,0,0);
    \coordinate (C) at (1,{sqrt(3)},0);
    
    \draw[thin, fill=orange!30] (A) -- (B) -- (C) -- cycle;

    \tdplotsetrotatedcoords{0}{0}{0}
    \begin{scope}[tdplot_rotated_coords]
      \coordinate (D) at (0.5, {sqrt(3)/4}, 0);
      \coordinate (E) at (1.5, {sqrt(3)/4}, 0);
      \draw[thick, black] (D) -- (E);
      
      \foreach \k in {1,...,41} {
        \pgfmathsetmacro{\xk}{0.5 + (\k-1) * (1 / 40)}
        \pgfmathsetmacro{\yk}{sqrt(3)/4}
        \pgfmathparse{1 / sqrt(15 * (\k / 42 * (1 - \k / 42)))}
        \edef\zk{\pgfmathresult}
        \coordinate (P\k) at (\xk, \yk, 0);
        \coordinate (Q\k) at (\xk, \yk, \zk);
        \draw[thin, gray] (P\k) -- (Q\k);
      }
      
      \pgfmathsetseed{3} 
      \foreach \j in {1,...,600} {
        \pgfmathsetmacro{\xrand}{rand} 
        \pgfmathsetmacro{\yrand}{rand * 0.05 - 0.02} 
        \pgfmathsetmacro{\xarcsin}{sin(deg(asin(\xrand)))} 
        \filldraw[purple] (1/2*\xarcsin + 1, \yrand + 0.433012, 0) circle (0.05pt);
      }
    \end{scope}

    \node at (D) [left] {\small{$\delta_0$}};
    \node at (E) [below] {\small{$\delta_1$}};
    \node at (0.9, 0.8/2, 0) [below] {\small{$\cS$}};
    \node at (0.8,-.1, 0) [below] {\small{$\cM_1(X)$}};
  \end{tikzpicture}
  \caption{\small{A schematic picture for the distributional limit law for empirical measures. The small dots around the set $\cS$ are empirical measures $e_n(x)$  of an ensemble of points in the phase space chosen randomly (w.r.t.\ an absolutely continuous probability measure) for some fixed large $n$. As $n$ increases, each individual dot oscillates between $\delta_0$ and $\delta_1$ (\cite{AarThaZwe2005}), but in such a way as to make together an asymptotic distribution on $\cS$ (Theorem \ref{thm:tha-dist-conv}). }} \label{fig:D1}
\end{figure}


  Let $\cM_1(M)$ denote the space of probability measures on a measure space $M$.
  The sequence $e_n:[0,1]\to\cM_1([0,1])$ induces the pushforward sequence
  $e_{n*}:\cM_1([0,1])\to\cM_1(\cM_1([0,1]))$.
   Since $(e_{n*}\lambda)(\cE)=\lambda(e_n\in\cE)$, we can view Theorem~\ref{thm:tha-dist-conv} equivalently as a statement about the limit of
$e_{n*}\lambda$ for all absolutely continuous measures $\lambda$:

\begin{corollary}
  Let $\omega\in\cM_1(\cM_1([0,1]))$ be the distribution of $\nu_{Z_{\alpha,\bar p}}$.  Then
\( \lim_{n\to\infty} e_{n*} \lambda = \omega \) for every absolutely continuous probability measure~$\lambda\in\cM_1([0,1])$.
\end{corollary}

\vspace{-.75ex}
Since $e_{n*}\lambda$ is a probability measure on the convex set $\cM_1([0,1])$, we can consider its expectation $\E (e_{n*}\lambda) =\int_{\cM_1([0,1])}\omega_0 \, de_{n*}\lambda \in \cM_1([0,1])$, which can be interpreted as a probability measure on $[0,1]$ satisfying $\E(e_{n*}\lambda) ( E ) = \int \omega_0(E) \, de_{n*}\lambda(\omega_0)$.

We have the following consequence of Theorem \ref{thm:dist-conv}:

\begin{corollary}\label{cor:tha-avg-bir}
Let $\lambda$ be an absolutely continuous probability measure on $[0,1]$.
Then

\noindent (a) 
  $\frac{1}{n} \sum_{ j = 0 }^{ n - 1 } f^j_*\lambda= \int e_n  \, d \lambda  = 
 \E(e_{n*}\lambda)$ for all $n\ge1$.

\noindent (b) The common limit of the sequences of measures in (a) is
given by
$\lim_{ n \to \infty } \E(e_{n*}\lambda)=  \nu_{ \bar{p} }$.
\end{corollary}

\subsubsection*{Beyond two neutral fixed points}
Our results generalise to any finite number of neutral fixed points $\xi_1,\dots,\xi_d$.
The proof of Theorem~\ref{thm:tha-mixing} requires no modifications.
The treatment of the arcsine law in~\cite{Tha2002,ThaZwe2006} assumes $d=2$ but this restriction is removed by Sera \& Yano~\cite{SerYan2019} using a different method. We use~\cite{SerYan2019} to generalise Theorem~\ref{thm:tha-dist-conv}.
The generalisations of Theorem~\ref{thm:tha-mixing} and~\ref{thm:tha-dist-conv} are stated in Section~\ref{sec:abstract} as Theorems~\ref{thm:mixing} and~\ref{thm:dist-conv} respectively.

The final main result in this paper is Theorem~\ref{thm:ae} which shows that for general $d\ge1$ and $\alpha\in(0,1)$, the set of limit points of the sequence of empirical measures $e_n$ is almost surely equal to the $(d-1)$-dimensional simplex
\[
  \cS \coloneqq \{ p_1\delta_{\xi_1}+\dots+ p_d\delta_{\xi_d}:p_1,\dots,p_d\ge0,\;p_1+\dots + p_d=1\}.
\]
Previous arguments in~\cite{AarThaZwe2005,CoaLuz2024} exploited connectedness of the set of limit points and were effective for $d\le 2$, $\alpha\in(0,1]$. Our method uses instead the full support of the arcsine law in~\cite{SerYan2019} combined with Theorem~\ref{thm:dist-conv}. However, this full support property fails for $\alpha=1$ which means that Theorem~\ref{thm:ae} remains open for $d\ge3$, $\alpha=1$.

The remainder of the paper is organised as follows.
In Section~\ref{sec:abstract}, we present an abstract setup which allows for any finite number of neutral fixed points, and we state our main results 
Theorem~\ref{thm:ae},~\ref{thm:dist-conv} and~\ref{thm:mixing}.
Also, we state and prove Corollaries~\ref{cor:avg-bir} and~\ref{cor:mixing} which generalise Corollaries~\ref{cor:tha-avg-bir} and~\ref{cor:tha-mixing} respectively.
Theorems~\ref{thm:ae},~\ref{thm:dist-conv} and~\ref{thm:mixing}
are proved in Section~\ref{sec:proofs}.
In Section~\ref{sec:examples}, we verify that our abstract setup includes the classes of intermittent maps in~\cite{CoaLuzMuh2023,Tha1980}.

\vspace{1em}
\paragraph{\textbf{Notation}}
We use ``big O'' and $\lesssim$ notation interchangeably, writing $a_n=O(b_n)$ or $a_n\lesssim b_n$ if there are constants $C>0$, $n_0\ge1$ such that $a_n\le Cb_n$ for all $n\ge n_0$. As usual, $a_n=o(b_n)$ means that $a_n/b_n\to0$ and $a_n\sim b_n$ means that $a_n/b_n\to1$.

\section{Abstract statement of results}
\label{sec:abstract}

\subsection{Maps with Gibbs-Markov first return maps}
\label{sec:gm}

In this section we recall some definitions and basic properties of maps with Gibbs-Markov first return maps needed to state our main assumptions.

Let $X$ be a compact metric space with Borel probability measure $m$ and
let $f:X\to X$ be a nonsingular measurable map, so $m(f^{-1}E)=0$ for every null set $E$.
We suppose that $f$ is ergodic, so for every measurable set $E$ with $f^{-1}E=E$ we have $m(E)=0$ or $1$. Also, we suppose that $f$ is
 conservative, which is equivalent to requiring that $f^nx\in E$ infinitely often for a.e.\ $x\in E$ for any positive measure set $E$.
(A standard reference for this material is~\cite[Ch 1]{Aar1997}.)

We say that the map $f$ together with a countable partition $\cP$ of $X$ into positive measure subsets is a \emph{Markov map} if the restriction of $f$ to each partition element is a measurable bijection onto a union of partition elements. 
For $\theta\in(0,1)$, we define the symbolic metric 
  $d_{\theta} (x,y) \coloneqq \theta^{s(x,y)}$ where 
  $s(x,y) \coloneqq \inf \{ n {\geq 0} : f^nx, f^ny \text{ lie in different elements of }\cP \}$.
It is assumed that $s(x,y)=\infty$ if and only if $x=y$ so that $d_\theta$ is a metric.

A Markov map $f$ is said to be \emph{topologically mixing} (w.r.t.\ the symbolic metric topology) if for every pair of open nonempty subsets
$U,V\subset X$, there 
exists an $N\ge1$ such that $f^n U \cap V \neq \emptyset$ for every $n \geq N$.
Equivalently, for every
$a,b \in \cP$ there exists an $N\ge1$ such that $f^n a \cap b \neq \emptyset$ for every $n \geq N$.

A standard approach in dynamical systems is \emph{inducing}, whereby one studies the complex dynamics of $f$ on $X$ by focusing on the ``nicer'' return dynamics to a well-chosen inducing domain $Y\subset X$.
We take 
$Y \subset X$ to be a union of partition elements and define the \emph{first return time} $\tau:Y\to\Z^+$ and \emph{first return map} $F: Y \to Y$ by setting 
\[
\tau (y) \coloneqq \inf \{ n \geq 1 : f^ny \in Y \}, \qquad
F(y) \coloneqq f^{ \tau (y) }y.
\]
(Conservativity ensures that $\tau$ is finite a.e.\ and hence that $F$ is well-defined a.e.)

Let $\cP_Y$ be the partition of $Y$ consisting of sets
$\big(\bigcap_{j=0}^{n-1} f^{-j}a_i\big)\cap f^{-n}Y$ where
$a_0,\dots,a_{n-1}\in\cP$ with $a_0\subset Y$ and $a_1,\dots,a_{n-1}\subset X\setminus Y$. Then $F$ is Markov with partition $\cP_Y$.

For the map $F$, we define the symbolic metric $d_\theta=\theta^s$ w.r.t.\ the partition $\cP_Y$. 
The map $F$ is \emph{Gibbs-Markov} if 
\begin{itemize}
  \item \textbf{finite images}: $\Card \{ F a : a \in \cP_{ Y } \} < \infty$;
  \item \textbf{bounded distortion}: there exists $\theta \in (0,1)$ and $C > 0$ such that the function $\log \frac{ d m } { dm \circ F }$ is $d_\theta$-Lipschitz on elements of $\cP_Y$.
\end{itemize}

Denote by $\cB_{\theta}(Y)$ the space of real-valued functions which are Lipschitz continuous on $Y$ with respect to the metric $d_{\theta}$. We recall the following classical result about the existence of invariant measures for maps with Gibbs-Markov first return maps.

\begin{lemma}
  \label{lem:gm}
  Suppose that $f : X \to X $ is a conservative ergodic Markov map with Gibbs-Markov first return map $F=f^\tau : Y \to Y $. Then $f$ preserves a unique (up to scaling) absolutely continuous $\sigma$-finite measure $\mu$, and $\mu$ is equivalent to $m$.

Moreover, the density $h=d\mu/dm$ is nonzero and bounded on $Y$ and
$h|_Y,\,(h|_Y)^{-1}\in\cB_\theta(Y)$. In particular, $\mu(Y)<\infty$.

Furthermore, $\mu ( X ) < \infty$ if and only if $\tau\in L^1(Y,m)$.
\end{lemma}

\begin{proof}
  Standard references include~\cite{AarDenUrb1993} and~\cite[Chapter~4]{Aar1997}.
\end{proof}

\subsection{Assumptions and notations}

Let $X$ be a compact metric space with Borel probability measure $m$.
Throughout, we suppose that $f:X\to X$ is a conservative, exact, Markov map
with Gibbs-Markov first return map $F=f^\tau:Y\to Y$, and
moreover that $f$ and $F$ are topologically mixing.
Let $\mu$ be the $f$-invariant absolutely continuous $\sigma$-finite measure in Lemma~\ref{lem:gm}.
We assume that $\mu(X)=\infty$.
Let $d\ge1$ be a positive integer.

\begin{description}
   \item[H1\label{itm:H1}]
   There exist fixed points $\xi_1,\dots,\xi_d \in X$ such that
   \[
     \mu\big(
     X\setminus \{B_{ \eps} (\xi_1) \cup \dots \cup  B_{\eps} ( \xi_d )\}
     \big)<\infty
     \quad\text{for all $\eps>0$.}
   \]
\end{description}
This assumption implies that neighbourhoods of the fixed points $\xi_1, \dots, \xi_d$ carry the infinite part of the mass. 

The next assumption ensures that the inducing set $Y$ \emph{dynamically separates} $\xi_1,\dots,\xi_d$ and that the excursion times from \( Y \) to neighbourhoods of the fixed points have certain tail distributions.

\begin{description}
  \item[H2\label{itm:H2}]
  There exist constants $\alpha\in(0,1]$, \( c_1,\dots,c_d>0 \)
  and a measurable partition \( X_1,\dots,X_d \) of \( X \setminus Y \)  
such that 
  \begin{enumerate}
    \item[(a)\label{itm:A2.a}]
    $\xi_k \in \Int X_k$, for $k = 1,\dots, d$;
    \item[(b)\label{itm:A2.b}]
    For \( k\neq \ell \), orbits cannot pass from $X_k$ to $X_{\ell}$ without first entering $Y$.
    Equivalently, 
    \[
      \{\tau=n\} = \bigcup_{k=1,\dots,d} \{ \tau^{ (k) } = n - 1 \}
      \quad\text{for $1\le n<\infty$}
    \]
    where $\tau^{(k)} (y) \coloneqq \Card \{ n \leq \tau (y) : f^n y \in X_k \} $, for $y\in Y$, $k = 1, \dots, d$;
    \item[(c)\label{itm:A2.c}]
    $\mu ( \tau^{(k)} > n ) \sim c_k n^{-\alpha}$ as $n\to\infty$ for $k=1,\dots,d$. 
  \end{enumerate}
\end{description}
Note that 
\begin{equation} \label{eq:tails}
  \mu ( \tau > n ) \sim c_{\tau} n^{-\alpha}
  \quad\text{where}\quad
  c_{\tau} \coloneqq c_1+\dots+c_d>0.
\end{equation}

\begin{remark}\label{rmk:H2}
  We expect that the results in this paper remain valid when the constants
  \( c_{ k } \) in~\ref{itm:H2}(c) are replaced by $c_k\ell(n)$ where $\ell$ is \emph{slowly varying}, i.e.\ $\lim_{ x \to \infty } \ell ( \lambda x ) / \ell ( x ) = 1$ for all $\lambda \in \R$. However, the proof of Theorem~\ref{thm:mixing} requires further calculations.
\end{remark}

Our third assumption is somewhat technical to state but is very mild.
Given $a_0,\dots,a_r\in\cP$, we define the $(r+1)$-cylinder
$[a_0,\dots,a_{r}]=\bigcap_{j=0}^{r}f^{-j}a_j$.
Let
$\cP_r$ denote the partition of $X$ into $(r+1)$-cylinders
and
let $\cP_r^*$ consist of $(r+1)$-cylinders $a=[a_0,\dots,a_r]\in\cP_r$ where
$a_0,\dots,a_{r-1}\not\subset Y$ and $a_r\subset Y$.
Given $\rho\in L^1(X,m)$ and $r\ge0$, define
\begin{equation} \label{eq:Q}
  Q_r^\rho:Y\to\R,  \qquad
  Q_r^\rho \coloneqq \sum_{a\in\cP_r^*} 1_{f^ra} \cdot \frac{\rho}{Jf^r}\circ f^r|_a^{-1}
\end{equation} 
where $J\! f^r \coloneqq \frac{ d m \circ f^{ r } }{ dm }$ denotes the Jacobian of $f^r$ with respect to $m$.

\begin{description}
  \item[H3\label{itm:H3}]
  There exists $\theta\in(0,1)$ and a dense subset $K(X)\subset L^1(X,m)$ such that
  $Q_r^\rho\in \cB_\theta(Y)$ for all $\rho\in K(X)$ and $r\ge0$.
\end{description}

Our final assumption is the following ``smooth tail'' estimate:

\begin{description}
  \item[H4\label{itm:H4}] If \( \alpha \in (0, 1 / 2] \), then \( \mu ( \tau = n ) = O ( n ^{ - ( 1+ \alpha  ) }) \). 
\end{description}

\subsection{Almost sure behaviour of the empirical measures}

We let $\cM_1(X)$ denote the space of Borel probability measures on $X$. Define the map
$e_n : X \to \cM_1 ( X )$
which maps each point $x\in X$ to its $n^{th}$ empirical measure $e_n(x)$ as given in~\eqref{eq:en}. 

Our first result concerns the almost sure behaviour of the empirical measures $e_n$. For $x \in X$, we let $\cL(x) \subset \cM_1(X)$ denote the set of weak-$*$ limit points of $e_n(x)$. Let $\cS_0$ be the simplex
\[
  \cS_0 \coloneqq \{ p \in [0,1]^d : p_1 + \dots + p_d  = 1 \}.
\]
For $p \in \cS_0$, let $\nu_{p} \coloneqq p_1 \delta_{ \xi_1 }  + \dots +  p_d \delta_{ \xi_d }$ be the corresponding convex combination of Dirac masses at the fixed points, and define
$$
\cS \coloneqq \{ \nu_{p} : p \in \cS_0 \},
$$
to be the set of all such convex combinations. 
With this notation in place we can state our first result.

\begin{manualtheorem}{A}\label{thm:ae}
  Suppose that \ref{itm:H1}--\ref{itm:H2} hold and that $\alpha \in (0,1)$. Then 
  $$
  \cL (x) = \cS
  \quad\text{for a.e.\ }x\in X.
  $$
  In particular, when $d\ge2$ there are no physical measures.
\end{manualtheorem}

\begin{remark}
  The inclusion $\cL (x) \subset \cS$ is elementary, see Lemma~\ref{lem:w*} below, so it is the reverse inclusion that is of interest.
  The case $d=1$ is trivial: for all $\alpha\in(0,1]$, we have $\cL(x)=\delta_{\xi_1}$ a.e.\ and hence
$\delta_{\xi_1}$ is a physical measure.  The case $d=2$ is also well-understood.
  Indeed, there are techniques~\cite{AarThaZwe2005,CoaLuz2024} for showing that $\delta_{\xi_k}\in\cL(x)$ a.e.\ for each $k$ and Theorem~\ref{thm:ae} then follows for $d=2$ by connectedness of $\cL(x)$. This approach includes the case $\alpha=1$.

  For $d\ge3$, Theorem~\ref{thm:ae} is completely new to the best of our knowledge.
  Our proof involves a different technique which works for all $d$ but only in the range $\alpha\in(0,1)$.
  We expect that Theorem~\ref{thm:ae} holds also for $d\ge3$, $\alpha=1$, but new ideas seem to be required.
\end{remark}

\subsection{Distributional convergence of the empirical measures}

Our second result concerns strong distributional convergence of the sequence $e_n$ of empirical measures.

\textbf{Notation:} 
Suppose that $Z_n$ is a sequence of measurable functions on $X$ taking values in some Borel space $M$ and that $Z$ is a random variable with distribution $\omega \in \cM_1 ( M)$ also taking values in $M$ (but not necessarily defined on $X$). 
Given $\lambda\in\cM_1(X)$,
we write $Z_n \to_\lambda Z$ if 
$\lim_{ n \to \infty }Z_{n*} \lambda = \omega$.

\begin{definition}
Let $\cM^{\mathrm{ac}}_1(X)$ denote the set of Borel probability measures on $X$ which are absolutely continuous with respect to $m$.
We say that
$Z_n$ \emph{converges strongly in distribution} to $Z$, and write
$Z_n\to_dZ$, if $Z_n\to_\lambda Z$ 
for all  $\lambda \in \cM^{\mathrm{ac}}_1 (X)$.
\end{definition}

Fix $\eps$ so that the neighbourhoods $B_{\eps}(\xi_k), \;k=1,\dots,d$ are disjoint.
As was the case for the Thaler maps in Theorem~\ref{thm:tha-dist-conv},
we first consider strong distributional convergence for occupation times
$S_n=(S_n^1,\dots,S_n^d):X\to[0,1]^d$ defined by
\begin{equation} \label{eq:Sn}
  S_n^k \coloneqq \sum_{ j = 0 }^{ n - 1 } 1_{ B_{\eps}(\xi_k) } \circ f^j, \quad \text{for } k = 1, \dots, d, \; n\ge1.
\end{equation}

Following~\cite{SerYan2019,Ser2020}, we consider a multiray generalisation of the classical arcsine law.
For \( \alpha \in (0,1) \), $p\in\cS_0$, let \( \zeta_{1},\dots, \zeta_{d} \) be independent \( [0,\infty) \)-valued random variables defined on a common probability space with one-sided \( \alpha \)-stable distribution characterised for $k=1,\dots,d$ by
\begin{equation}\label{eq:def-zeta}
  \E \exp( - t \zeta_k ) = \exp ( -t^{\alpha} p_k ), \quad t > 0 .
\end{equation}
Define the $\cS_0$-valued random variable
\begin{equation}
  \label{eq:def-Z}
  Z_{\alpha, p}=(Z_{\alpha,p}^{(1)},\dots,Z_{\alpha,p}^{(d)}) \coloneqq \frac{1}{\zeta_1+\dots+\zeta_d} ( \zeta_1, \dots, \zeta_d).
\end{equation}
When $\alpha=1$, we define $Z_{1, p}\equiv  p$.

\begin{remark}
  \label{rem:arcsine-density}
  In the special case that \( d = 2 \), one has that \( Z_{\alpha, p}^{(2)} = 1 - Z_{\alpha, p}^{(1)} \) and for $\alpha\in(0,1)$ the distribution of \( Z_{\alpha, p}^{(1)} \) admits the continuous density in~\eqref{eq:arcsine-density}
  (identifying $ p$ with $ p_1$).
\end{remark}

\begin{remark} The first moment of the multiray arcsine law 
is given by $\E Z_{\alpha,p}=p$.
This can be verified using the
double Laplace formula~\cite[Proposition~2.6]{SerYan2019},
\[
\int_0^\infty e^{-qt} \E\Big( \exp\Big\{-t\sum_{k=1}^d \lambda_k\zeta_k\Big\}\Big)\,dt
=\frac{\sum_{k=1}^d p_k(q+\lambda_k)^{\alpha-1}}
{\sum_{k=1}^d p_k(q+\lambda_k)^{\alpha}}, \quad
\lambda=(\lambda_1,\dots,\lambda_d)\in[0,\infty)^d,\; q>0,
\]
as derived in~\cite[Proposition~3.6]{Yan2017}.
Differentiating w.r.t.\ $\lambda_k$ and setting $\lambda=0$ yields $\E Z_{\alpha,p}^{(k)}=p_k$.
\end{remark}

We can now recall the arcsine law of~\cite{SerYan2019,Ser2020}.
Recall from~\eqref{eq:tails} that $c_\tau=\sum_{k=1}^d c_k$. 
Set 
\[
  \bar{p} = ( \bar p_1,\dots, \bar p_d)
  \coloneqq c_\tau^{-1}( c_1 , \dots,  c_d) \in\cS_0.
\]

\begin{theorem}\label{thm:Sera}
  Suppose that \ref{itm:H1}--\ref{itm:H2} hold. Then 
  $
  \frac1n S_n \to_d Z_{\alpha,\bar p}.
  $
\end{theorem}

\begin{proof}
  It suffices to check that \ref{itm:H1}--\ref{itm:H2} imply \cite[Assumptions 2.1, 2.2 and 2.6]{Ser2020} as the result then follows from \cite[Corollary 4.2 and Theorem 3.3]{Ser2020}. 
  Let \( F:Y\to Y \) be the Gibbs-Markov first return map in  \ref{itm:H2}. Notice that \ref{itm:H2}(b) yields \cite[Assumption 2.1]{Ser2020}, and then \ref{itm:H2}(c) together with \cite[Lemma 2.4]{Ser2020} gives \cite[Assumption 2.2]{Ser2020}.
  Finally, as \( F \) is topologically mixing it is exponentially continued fraction mixing (see for example \cite[Section~4]{Aar1997}) which immediately implies \cite[Assumption 2.6]{Ser2020}.
\end{proof}

To each $\cS_0$-valued random variable $Z$, we can associate the $\cS$-valued random variable $\nu_Z=Z_1\delta_{\xi_1}+\dots+Z_d\delta_{\xi_d}$.
Notice that
$\E\nu_Z=\sum_{k=1}^d \E Z_k\delta_{\xi_k}
=\sum_{k=1}^d (\E Z)_k\delta_{\xi_k}
=\nu_{\E Z}$.
In particular, $\E \nu_{Z_{\alpha,p}}=\nu_p$.

\begin{manualtheorem}{B}\label{thm:dist-conv}
  Suppose that \ref{itm:H1}--\ref{itm:H2} hold. Then $ e_n \to_d \nu_{Z_{\alpha,\bar p}}$. 
  Equivalently,
  \(
  \lim_{ n \to \infty } e_{n*} \lambda = \omega 
  \)
  for all \( \lambda \in \cM_1^{ \mathrm{ac} } ( X )\)
  where $\omega\in\cM_1(\cS)$ is the distribution of $\nu_{Z_{\alpha,\bar p}}$.
\end{manualtheorem}

As in the introduction, we obtain the following consequence:
\begin{corollary}\label{cor:avg-bir}
Let $\lambda\in \cM_1^{ \mathrm{ac} } ( X )$.
Then

\noindent (a) 
  $\frac{1}{n} \sum_{ j = 0 }^{ n - 1 } f^j_*\lambda= \int e_n  \, d \lambda  = 
 \E(e_{n*}\lambda)$ for all $n\ge1$.

\noindent (b) The common limit of the sequences of measures in (a) is
given by
$\lim_{ n \to \infty } \E(e_{n*}\lambda)=  \nu_{ \bar{p} }$.
\end{corollary}

\begin{proof}
Let $i:\cM_1(X)\to \cM_1(X)$ be the identity map.
Then
\[
\E(e_{n*}\lambda)=\int_{\cM_1(X)}\omega_0\,d (e_{n*}\lambda)(\omega_0)
=\int_{\cM_1(X)}i(\omega_0)\,d (e_{n*}\lambda)(\omega_0)
=\int_X i\circ e_n(x)\,d \lambda(x)
=\int_X e_n\,d \lambda.
\]

Next,
\[
\bigg(\int_X \delta_{f^jx}\,d\lambda(x)\bigg)(E)
=
\int_X \delta_x(E)\,df_*^j\lambda(x)
=
\int_X 1_E \, df_*^j\lambda=(f_*^j\lambda)(E),
\]
so
  $\int_X \delta_{f^jx} \,d\lambda(x)=f^j_*\lambda$.
Hence
\[
    \int_X e_n  \, d \lambda= \frac{1}{n} \sum_{ j = 0 }^{ n - 1 } \int_X \delta_{f^jx} \,d\lambda(x).
  =\frac{1}{n} \sum_{ j = 0 }^{ n - 1 } f^j_*\lambda.
\]

\noindent(b)
  Since $\E$ acts continuously from $\cM_1(\cM_1(X))$ to $\cM_1(X)$, it follows from 
Theorem~\ref{thm:dist-conv} that
  $\lim_{ n \to \infty } \E(e_{n*}\lambda)= \E \nu_{Z_{\alpha,\bar p}} = \nu_{ \bar{p} }.$
\end{proof}

\subsection{Existence and uniqueness of natural measures}

Our third result concerns weak-$*$ convergence of pushforwards $f^n_*\lambda$
of absolutely continuous probability measures $\lambda\in\cM_1^{\mathrm ac}(X)$.

\begin{manualtheorem}{C}\label{thm:mixing}
  Suppose that \ref{itm:H1}--\ref{itm:H4} hold with $\alpha\in(0,1]$.
  Then 
  $$
  \lim_{ n \to \infty} f_*^n \lambda = \nu_{\bar{p}}
  \quad\text{for all }\lambda\in \cM_1^{\mathrm ac}(X).
  $$
\end{manualtheorem}

As in the introduction, we obtain the following result on decay of correlations:

\begin{corollary}\label{cor:mixing}
  Suppose $\varphi:X\to\R$ is bounded, measurable, and continuous at each $\xi_1,\dots, \xi_d$.
  Then
  $$
  \lim_{n\to\infty}	\int \psi \cdot \varphi \circ f^n \, d m = \int \psi \, d m \int \varphi \, d \nu_{\bar{p}}
  \quad\text{for all } \psi \in L^1 (X, m  ). 
  $$
  Equivalently,
  $
  \lim_{n \to \infty } \int  \psi \cdot \varphi \circ f^n \, d \mu = \int  \psi \, d \mu \int \varphi \, d \nu_{ \bar{p} }$
  for all  $\psi \in L^1 (X, \mu  ). 
  $
\end{corollary}

\begin{proof}
  We begin by proving that the first limit holds.

  By Theorem~\ref{thm:mixing}, $f_*^n\lambda\to\nu_{\bar p}$ weak-$^*$, so by definition  
  $\int\varphi\,df_*^n\lambda\to\int\varphi\,d\nu_{\bar p}$ for all continuous $\varphi:X\to\R$.
  By~\cite[Theorem~25.7]{Bil1995},
  $\int\varphi\,df_*^n\lambda\to\int\varphi\,d\nu_{\bar p}$ for 
  all bounded measurable functions
  ${\varphi:X\to\R}$ that are continuous except on a set of $\nu_{\bar p}$-measure zero, which is precisely the class of observables in the statement of the corollary.

  Fix such a $\varphi$.
  Let $\psi\in L^1(X,m)$ with $\int\psi\,dm\neq0$ and write $\psi=\rho\int\psi\,dm$ where $\rho\in L^1(X,m)$ and $\int\rho\,dm=1$.
  Let $d\lambda=\rho\,dm$. 
  Then
  $$
  (\textstyle{\int} \psi \, dm )^{-1}
  \int \psi \cdot \varphi \circ f^n \, dm = 
  \int \rho \cdot \varphi \circ f^n \, dm = 
  \int \varphi \circ f^n \, d\lambda = 
  \int \varphi \, df_*^n\lambda \to \int\varphi\,d\nu_{\bar p}.
  $$
  The case $\int\psi\,dm=0$ is dealt with by approximating, concluding the proof of the first limit.

  If $\psi\in L^1(X,\mu)$, then $\psi h\in L^1(X,m)$ and the second limit follows from the first. Similarly, the first limit follows from the second.
\end{proof}

\section{Proofs}
\label{sec:proofs}

In this section, we prove the three main theorems in this paper.
It turns out to be convenient to prove them in the order B, A, C.

\subsection{Ergodic theorem}

In this subsection, we recall the ergodic theorem for infinite measure systems and derive a consequence for the systems studied in this paper.

\begin{lemma}
  \label{lem:aar}
  Assume that \ref{itm:H1} holds.
  If $\varphi \in L^1 (X, \mu )$, then
  $\lim_{n\to\infty} \frac{1}{n} \sum_{ j =0 }^{ n - 1 } \varphi \circ f^j  = 0$ almost everywhere.
\end{lemma}

\begin{proof}
  This is an easy consequence of the Hopf ratio ergodic theorem, see for example~\cite[Exercise 2.2.1]{Aar1997}.
\end{proof}

\begin{lemma} \label{lem:w*}
  Assume that \ref{itm:H1} holds.
  Then $\cL (x) \subset \cS$ for almost every $x\in X$.
\end{lemma}

\begin{proof}
  We provide the details for completeness (cf.~\cite[p.~963]{AarThaZwe2005}).

  For $\eps>0$, define  $X^{(\eps)} \coloneqq X\setminus \{B_{ \eps} (\xi_1) \cup \dots \cup  B_{\eps} ( \xi_d )\}$.
  By \ref{itm:H1}, $\mu(X^{(\eps)})<\infty$ for all ${\eps>0}$.
  Choose $\psi:X\to[0,1]$ continuous and supported in $X^{(\eps/2)}$ such that $\psi|X^{(\eps)}\equiv1$.
  In particular, $\int_X\psi\,d\mu<\infty$ so, by Lemma~\ref{lem:aar}, there exists $X'\subset X$ with $\mu(X\setminus X')=0$ such that $\lim_{n\to\infty}\frac{1}{n}\sum_{j=0}^{n-1}\psi(f^jx)=0$
  for all $x\in X'$.

  Suppose that $x\in X'$ and $\omega_0\in \cL(x)$, and choose a subsequence $n_i$ such that $e_{n_i}(x)\to \omega_0$. Then $\frac{1}{n_i}\sum_{j=0}^{n_i-1}\psi(f^jx)\to\int_X\psi\,d\omega_0$,
  yielding
  $\omega_0(X^{(\eps)})\le\int_X\psi\,d\omega_0=0$.
  Hence, elements of $\cL(x)$ are supported in
  $B_\eps(x_1)\cup\dots\cup B_\eps(x_d)$ for all $x\in X'$.
  Since $\eps>0$ can be chosen arbitrarily small, the result follows.
\end{proof}

\begin{corollary}\label{cor:w*}
  Assume that \ref{itm:H1} holds.
  Suppose that $e_{n_k}\to_m Z$ for some $\cM_1(X)$-valued random variable $Z$ and some subsequence $n_k$. Then $Z$ takes values in $\cS$.
\end{corollary}

\begin{proof}
  Let $\psi:\cM_1(X)\to\R$ be continuous and supported in $\cM_1(X)\setminus \cS$.  By the dominated convergence theorem,
  $$
  \int_{X} \psi(e_{n_k} (x) )\,dm (x) 
  \to \E \, \psi(Z) 
  .
  $$
  But $\psi(e_{n_k}) \to 0$ a.e.\ by Lemma~\ref{lem:w*}. Applying the dominated convergence theorem once more, $\E\, \psi(Z)=0$, and the result follows.
\end{proof}

\begin{remark} \label{rem:w*}
  Corollary~\ref{cor:w*} shows that the only possible distributional limit points of the sequence $e_n$ are random variables of the form $\nu_Z$  where $Z$ is an $\cS_0$-valued random variable.
\end{remark}

We recall also the following consequence of exactness.

\begin{lemma} \label{lem:exact}
Assume that~\ref{itm:H1} holds. Let $\lambda\in \cM_1^{\mathrm{ac}}(X)$. 
Suppose that $f_*^{n_k}\lambda\to \omega_0$ as $k \to \infty$ for 
some $\omega_0\in \cM_1(X)$
and some subsequence $n_k$.
Then $\omega_0\in\cS$.
\end{lemma}

\begin{proof}
We show that $\omega_0(E)=0$ for all $E\subset X$ with $\mu(E)<\infty$. 
The result then follows from~\ref{itm:H1}.
Our argument follows~\cite[Appendix~A]{BonLen2021}.

Write $\rho=d\lambda/dm$ and recall that $h=d\mu/dm$.
Then 
\[
f_*^n\lambda(E)=\int_X 1_E\circ f^n\,d\lambda=
\int_X 1_E\circ f^n\cdot \rho\,dm=
\int_X 1_E\circ f^n\cdot \tilde\rho\,d\mu,
\]
where $\tilde \rho=\rho h^{-1}$ satisfies $\int_X\tilde\rho\,d\mu=\int_X\rho\,dm=1$.
Let $A\subset X$ with $\mu(A)\in(0,\infty)$. Then
$g\coloneqq \tilde\rho-\mu(A)^{-1}1_A\in L^1(X,\mu)$ with $\int_X g\,d\mu=0$.
Hence $\lim_{n\to\infty}\int_X 1_E\circ f^n \cdot g\,d\mu=0$ by Lin's Theorem~\cite{Lin1971}.
Writing
\[
\int_X 1_E\circ f^n\cdot \tilde\rho\,d\mu=
\int_X 1_E\circ f^n\cdot g\,d\mu+
\int_X 1_E\circ f^n\cdot \mu(A)^{-1}1_A\,d\mu,
\]
we deduce that
\[
\limsup_{k\to\infty} 
f_*^{n_k}\lambda(E)
\le \limsup_{k\to\infty}\Big|\int_X 1_E\circ f^{n_k}\,d\mu\Big|\, \mu(A)^{-1}
= \mu(E) \mu(A)^{-1}
\]
The result follows since $\mu(E)<\infty$ and $\mu(A)\in(0,\infty)$ is arbitrarily large.
\end{proof}

\subsection{Proof of Theorem \ref{thm:dist-conv}}

Fix $\eps>0$ such that $B_\eps(\xi_k)$ are disjoint for $k=1,\dots,d$ and
define $S_n:X\to[0,1]^d$ as in~\eqref{eq:Sn}.

\begin{lemma}
  \label{lem:second}
  Let $Z$ be a random variable with values in $\cS_0$.
  Then 
  $e_n\to_m \nu_Z$ if and only if $n^{-1}S_n\to_m Z$.
\end{lemma}

\begin{proof}
  Define
  \[
    \pi : M_1(X)\to [0,1]^d, \qquad
    \pi(\omega_0) \coloneqq
    (\omega_0(B_{\eps}(\xi_1)),\dots, \omega_0(B_{\eps}(\xi_d)).
  \]
  Note that $\pi$ restricts to the natural identification $\nu_Z\mapsto Z$  between $\cS$ and $\cS_0$. Also, $\pi$ is continuous at elements in $\cS$ and satisfies $\pi ( e_n ) = \frac1n S_n$.

  In particular, if $e_{n_k}\to_m \nu_Z$ for some subsequence $n_k$, then it is an immediate consequence of the continuity of $\pi$ at $\nu_Z$ and the continuous mapping theorem that
  $\frac1n S_{n_k}=\pi (e_{n_k}) \to_m \pi(\nu_Z)=Z$
  completing the proof in one direction.

  The converse follows by a standard probabilistic argument. Suppose that $\frac1n S_n\to_m Z$. Let $\hZ$ be a distributional limit point for $e_n$, so $e_{n_k}\to_m \hZ$ for some subsequence $n_k$. By Remark~\ref{rem:w*}, $\hZ=\nu_A$ for some random variable $A$ with values in $\cS_0$. By what we just proved, $\frac1n S_{n_k}\to_m A$. Hence, $A=Z$, so $\hZ=\nu_Z$ is the unique distributional limit point for $e_n$. This means that $e_n\to_m \nu_Z$.
\end{proof}

\begin{proof}[Proof of Theorem~\ref{thm:dist-conv}]
  By Theorem~\ref{thm:Sera}, $n^{-1}S_n\to_m Z_{\alpha,\bar p}$.
  Hence by Lemma~\ref{lem:second}, 
  $e_n \to_m \nu_{Z_{\alpha,\bar p}}$.

  Let \(d_W\) denote the \emph{Wasserstein distance} on \(\cM_1(X)\),
  \[
    d_W ( \omega_0, \omega_0' ) \coloneqq
    \sup_{ \varphi \in \Lip_1 } \Big| \int \varphi \, d\omega_0 -
    \int \varphi \, d\omega_0' \Big|,
  \]
  where $\Lip_1 = \{ \varphi : X \to \R : \Lip \varphi+\|\varphi\|_\infty \le1 \}$  and $\Lip \varphi$ denotes the smallest Lipschitz constant of \(\varphi:X\to\R\). Recall that \(d_W\) induces the weak-$*$ topology on \(\cM_1 ( X)\).
  Also, 
  $ d_W ( e_n \circ f, e_n )\le \frac2n$, so
  the functions $e_n : X \to \cM_1(X)$ satisfy the ``asymptotic invariance'' condition
  $ d_W ( e_n \circ f, e_n )\to_m0$.
  Hence, we may apply \cite[Theorem~1]{Zwe2007} to deduce from
  $e_n\to_m \nu_{Z_{\alpha,\bar p}}$ that
  $e_n\to_d \nu_{Z_{\alpha,\bar p}}$.
\end{proof}

\subsection{Proof of Theorem~\ref{thm:ae}}

We first show that the limiting random variable \( \nu_{Z_{\alpha,\bar p}} \) in Theorem \ref{thm:dist-conv} has full support in \( \cS \) when \( \alpha \in (0,1) \).
\begin{lemma}\label{lem:sup-hat-omega}
  Let \( \alpha \in (0,1) \), $p\in\cS_0$ with $p_i>0$ for all $i$. Then
  \(
  \PP(\nu_{Z_{\alpha, p}}\in B_{ \eps }( \nu ) )>0\)
  for all  \(\eps > 0\), \(\nu \in \cS.  \)
\end{lemma}

\begin{proof}
  Let \( \zeta_1,\dots,\zeta_d \) be the independent,  \( \alpha \)-stable random variables that appear in the definition \eqref{eq:def-Z} of \( Z_{\alpha, p} \).
  Recall~\cite[Lemma 1.1, Proposition 3.2]{Nol2020} that as the \( \zeta_k \) are non-negative \( \alpha \)-stable random variables with Laplace transforms given by \eqref{eq:def-zeta}, their distributions are fully supported and continuous on the half line \( [0, \infty) \).

  For every \( q \in \cS_0\), \(\eps > 0 \) the set 
  \[
    U \coloneqq \left\{ x \in [0,\infty)^d : 
      \frac{x_k}{x_1+\dots + x_d} \in (q_k -\eps, q_k + \eps ), \text{ for all } k=1,\dots, d\right\}
  \]
  contains a nonempty open rectangle \( \prod_{k=1}^d (a_k, b_k) \). By definition of \( Z_{\alpha, p} \) and the independence of the \( \zeta_k \), 
  \begin{align*}
    \PP ( Z_{\alpha, p} \in B_{\eps} (q) )
    &= \PP ( (\zeta_1, \dots, \zeta_d) \in U ) \\
    &\geq \PP ( \zeta_k \in (a_k,b_k) : k = 1,\dots, d )
      = \prod_{k=1}^d \PP ( \zeta_k \in (a_k,b_k) ) > 0.
  \end{align*}
  It follows that $\PP(\nu_{Z_{\alpha,p}}\in B_\eps(\nu_q))>0$.
\end{proof}

\begin{proof}[Proof of Theorem~\ref{thm:ae}]
  By Lemma~\ref{lem:w*}, \( \cL (x) \subset \cS \) for almost every \( x \). To prove the converse, let \( \nu \in \cS \) and consider the function
  \[
    \varphi (x) \coloneqq \liminf_{ n\to\infty } d ( e_n (x), \nu ),
  \]
  where \( d \) is any metric metrising the weak-\( * \) topology. 
  We will show that \( \varphi = 0  \) almost everywhere and so \( \nu \in \cL (x) \) for almost every \( x \).
The result then follows since \( \nu \in \cS \) is arbitrary.

  Notice that \( \varphi \) is invariant for \( f \) and so, by ergodicity, must be almost everywhere equal to some constant \( c \geq 0 \).
  Suppose for contradiction that \( c > 0 \), and let \( 0 < \eps < c \). It then follows that  \( m \{ x : e_n (x) \in B_{\eps} ( \nu ) \text{ for infinitely many $n$} \} = 0 \). 

  By Theorem \ref{thm:dist-conv}, $e_n\to_m \nu_{Z_{\alpha,\bar p}}$ with
$\bar p_i>0$ for all $i$.
  By the Portmanteau lemma together with Lemma~\ref{lem:sup-hat-omega},
  \[
    \liminf_{ n \to \infty }  m ( e_n \in B_{\eps} ( \nu ) ) 
    \geq \PP(\nu_{Z_{\alpha,\bar p}}\in B_{\eps} ( \nu ))>0.
  \]
  Hence
  \begin{align*}
    m \{ x : e_n (x) \in B_{\eps} ( \nu ) \text{ for infinitely many $n$} \} 
    &= m \Big( \bigcap_{ n = 1 } \bigcup_{ \ell \geq n } \{ e_\ell \in B_{\eps} ( \nu ) \} \Big) \\
    &= \lim_{n\to\infty} m \Big( \bigcup_{ \ell \geq n } \{ e_\ell \in B_{\eps} ( \nu ) \} \Big) \\
    &\geq \liminf_{ n \to \infty }  m ( e_n \in B_{\eps} ( \nu ) ) 
      > 0,
  \end{align*}
  where we have used the fact that the sequence of sets 
  $A_n = \bigcup_{\ell \geq n}^{ \infty} \{ e_\ell \in B_{\eps} ( \nu ) \}$ is decreasing: $A_{n+ 1} \subset A_n$.
  This contradicts our assumption that \( c > 0 \) and so \( \varphi ( x ) = 0 \) almost everywhere. 
\end{proof}

\subsection{Proof of Theorem \ref{thm:mixing}}
\label{sec:mixing}

Let $f$ satisfy~\ref{itm:H1}--\ref{itm:H4}.

We let \( L : L^1 (X, \mu ) \to L^1 (X, \mu ) \) be the \emph{transfer operator} for \( f \) defined by the relation
\[
  \int_X Lv \cdot w  \, d\mu  = \int_X  v  \cdot w\circ f
  \, d\mu, \quad \text{for all } v \in L^1 (X, \mu ), \; w \in L^{\infty}(X).
\]
Define for $n\ge0$,
\[
  T_n:L^1(Y,\mu|_Y)\to L^1(Y,\mu|_Y), \qquad
  T_nv=1_Y L^n(1_Yv).
\]
We recall some results from \cite{Gou2011,MelTer2012} which describe the asymptotic behaviour of \( T_n \) acting on the space $\cB_\theta(Y)$ in Section~\ref{sec:gm}.

\begin{theorem}
  \label{thm:mixing-infinite}
  Let \( c_{\tau}\in(0,\infty) \) be as in~\eqref{eq:tails}.
  Then, for every \( v \in \cB_{\theta}(Y) \),
  \begin{alignat*}{2}
    c_{\tau} \log  n\,  T_n v  & \to \int_Y v \, d\mu  , &
    & \alpha=1 
    \\
    c_{\tau} n^{1 - \alpha}  T_n v  & \to \tfrac{1}{\pi} \sin  \pi \alpha \int_Y v \, d \mu, & \qquad
    & \alpha\in(0,1)
  \end{alignat*}
  uniformly on $Y$ as $n\to\infty$.
\end{theorem}

\begin{proof}
  By assumption, $F=f^\tau:Y\to Y$ is a Gibbs-Markov map and
  $\mu ( \tau > n ) \sim c_\tau n^{-\alpha}$.
  Also, the underlying conservative ergodic map $f:X\to X$ is topologically mixing.

  In the range $\alpha\in(\frac12,1]$ we can apply~\cite[Theorem 2.1, Proposition~11.4]{MelTer2012}.

  In the range $\alpha\in(0,\frac12]$, we use~\cite[Theorem 1.4]{Gou2011}.
  The hypotheses in~\cite{Gou2011} are stated slightly differently than in~\cite{MelTer2012}. In the notation of these papers, the essential differences are as follows: (i) There is a stronger assumption on $\|R_n\|$ which holds since
  $\mu(\tau=n)=O(n^{-(\alpha+1)})$; (ii) There is the requirement that $R(1)$ has no eigenvalues on the unit circle besides $1$ which holds since $F$ is topologically mixing (see for example~\cite[Theorem~1.6]{AarDen2001}).
\end{proof}

Define
\[
  Y_{ 0 } \coloneqq Y; \quand
  Y_r \coloneqq f^{-r}Y\setminus E_{r-1},
  \quad \text{for } r \ge 1.
\]
where $E_r\coloneqq \bigcup_{i=0}^{r}f^{-i}Y$.

Recall that $h=\frac{d\mu}{dm}$ denotes the density of the invariant measure $\mu$. 
Let $\lambda \in \cM^{\mathrm{ac}}_1 (X)$ with density $\rho = \frac{ d \lambda }{ dm }$.
Let $K(X)$ be the dense subset of $L^1(X,m)$ in~\ref{itm:H3}.

\begin{proposition} \label{prop:K}
  If \( \rho \in K(X) \),  then $L^r(\rho h^{-1}1_{Y_r}) \in\cB_\theta(Y)$ for all $r\ge0$.
\end{proposition}

\begin{proof}
  Note first that $L^r(\rho h^{-1}1_{Y_r})$ is supported in $Y$ by definition of $Y_r$.
  Let $M:L^1(X)\to L^1(X)$ denote the transfer operator for the reference measure $m$, so $L=h^{-1}Mh$. As usual, it follows from the definition $\int M^rv\cdot w\,dm=\int v\cdot w\circ f^r\,dm$ and change of variables that
  \[
    M^rv=\sum_{a\in\cP_r}1_{f^ra} \cdot \frac{v}{J^rf}\circ f^r|_a^{-1}.
  \]
  Hence, recalling~\eqref{eq:Q},
  \begin{align*}
    L^r(\rho h^{-1}1_{Y_r}) = h^{-1}M^r(\rho 1_{Y_r})
    & = h^{-1} \sum_{a\in\cP_r}1_{f^ra} \cdot 
      1_{Y_r}\circ f^r|_a^{-1} \cdot
      \frac{\rho}{Jf^r}\circ f^r|_a^{-1}
    \\ & = h^{-1} \sum_{a\in\cP_r^*}1_{f^ra} \cdot 
         \frac{\rho}{Jf^r}\circ f^r|_a^{-1}
         =h^{-1}Q_r^\rho.
  \end{align*}
  By~\ref{itm:H3},
  $Q_r^\rho\in \cB_\theta(Y)$, while $(h|_Y)^{-1}\in\cB_\theta(Y)$ by Lemma~\ref{lem:gm}.  Hence $h^{-1}Q_r^\rho\in \cB_\theta(Y)$.
\end{proof}

For \( r \ge 0 \), define 
\(
\rho_r\coloneqq \rho h^{-1} 1_{E_r}. 
\)
Notice that
\begin{equation}\label{eq:rho-r}
  \rho_{ r } = \sum_{ j = 0 }^{ r } \rho h^{ -1 } 1_{ Y_{ j } }.
\end{equation}

\begin{corollary}
  \label{cor:mixing-infinite}
  For all \( \rho \in K(X) \) and $r\ge0$,
  \begin{alignat*}{2}
    c_{\tau} \log  n\,  L^n \rho_r  & \to \int_X \rho_r \, d\mu  , &
    & \alpha=1 
    \\
    c_{\tau} n^{1 - \alpha}  L^n \rho_r  & \to \tfrac{1}{\pi} \sin  \pi \alpha \int_X \rho _r \, d \mu, & \qquad
    & \alpha\in(0,1)
  \end{alignat*}
  uniformly on $Y$ as $n\to\infty$.
\end{corollary}

\begin{proof}
  We give the details in the case $\alpha=1$. The case $\alpha\in(0,1)$ is similar.

  For $n\ge r\ge0$, equation \eqref{eq:rho-r} gives 
  $$
  1_YL^n \rho_r
  =\sum_{j=0}^r 1_YL^{n}(\rho h^{-1}1_{Y_j})
  =\sum_{j=0}^r T_{n-j}L^j(\rho h^{-1}1_{Y_j}).
  $$
  Using again~\eqref{eq:rho-r} and 
the fact that \( \int_{ X }  \rho h^{ -1 } 1 _{  Y_{ j } } \,d \mu  = \int_{ Y }  L^{ j } ( \rho h^{ -1 } 1 _{  Y_{ j } } ) \,d \mu \), we obtain that on \( Y \)
  \[
    \begin{split}
      &c_\tau \log n\, L^n \rho_r-\int_X \rho_r\,d\mu \\
      &\quad
        = \sum_{j=0}^r \frac{\log n}{\log(n-j)}
        \Big\{c_\tau \log(n-j)T_{n-j}L^j(\rho h^{-1}1_{Y_j})
        -\int_Y L^j(\rho h^{-1}1_{Y_j})\,d\mu\Big\} \\
      & \quad\quad
        +\sum_{j=0}^r \Big\{\frac{\log n}{\log(n-j)}-1\Big\}\int_X \rho h^{-1}1_{Y_j}\,d\mu.
    \end{split}
  \]
  The result follows by Theorem~\ref{thm:mixing-infinite} and 
  Proposition~\ref{prop:K}.
\end{proof}

Now, let $X_1,\dots,X_d$ be the partition of $X \setminus Y$ in \ref{itm:H2}. 

\begin{lemma}\label{lem:decomposition}
  For all $\rho\in K(X)$, $n>r\ge0$, $k = 1,\dots, d$, 
  $$
  \int_X 1_{X_k} \, d f_{*}^n ( \lambda |_{E_r}) 
  = \sum_{ j = 1 }^n\int_{ \{ \tau^{(k)} \ge j  \}}  1_Y L^{ n - j } 
  \rho_r  \, d\mu.
  $$
\end{lemma}

\begin{proof}
  Suppose that $x\in E_r$ and that $f^n x \in X_k$ for some \( k = 1 ,\dots, d \). Then $x$ must have made its last return to $Y$ at some time $n - j$ for some $1 \leq j \leq n$. Moreover, by \ref{itm:H2}(b), $f^mx\in X_k$ for $n-j+1\le m\le n$.
  Hence,
  \begin{align*}
    \{ x \in E_r: f^n x \in X_k \} &= 
                                     \bigcup_{ j = 1 }^n \{x\in E_r: f^{n - j }x \in Y \text{ and } f^m x\in X_k,\, n-j+1\le m\le n \} \\
                                   &= \bigcup_{ j = 1 }^n \{x\in E_r: f^{n - j }x \in Y \text{ and } \tau^{(k)}(f^{ n - j}x) \ge  j   \}.
  \end{align*}
  As this is a disjoint union, 
  \begin{align*}
    \int_X 1_{ X_k }\, d f_*^n (\lambda|_{E_r})
    &= \int_X 1_{ X_k }  \circ f^n \cdot 1_{E_r}\cdot \rho \cdot h ^{-1} \, d\mu \\
    &=  \sum_{ j = 1 }^n \int_X (1_{ \{ \tau^{(k)} \ge j  \}}1_Y) \circ f^{ n - j } \cdot 1_{E_r} \cdot \rho \cdot h ^{-1} \, d\mu \\
    &= \sum_{ j = 1 }^n \int_{ \{ \tau^{(k)} > j - 1 \} } 1_Y L^{ n - j }  \rho_r  \, d\mu.
  \end{align*}
\end{proof}

\begin{proposition}\label{prop:mixing-main}
  For all $\rho\in K(X)$, $r\ge0$, $k=1,\dots, d$, 
  \[
    \lim_{ n \to \infty} \int_X 1_{ X_k } df_{*}^n (\lambda|_{E_r}) = \bar p_k\lambda\left(\bigcup_{i=0}^r f^{-i}Y\right).
  \]
\end{proposition}

\begin{proof}
  Note that $\int\rho_r\,d\mu=\int_{E_r}\rho h^{-1}\,d\mu=\lambda(E_r)$.

  First, we consider the case $\alpha \in (0,1)$. Set
  \[
    \eps_{ n } \coloneqq
    \sup_{ Y }\Big| L^n \rho_r   - \frac{ d_{\alpha} \lambda (E_r) }{ n^{ 1 - \alpha }}\Big|,
  \]
  where \( d_{\alpha} = \frac{ \sin  \pi \alpha } {\pi c_{\tau }} \).  By Corollary~\ref{cor:mixing-infinite},
  \(
    \eps_n = o ( n^{  \alpha - 1} ),
  \)
  Hence, by Lemma~\ref{lem:decomposition},
  \begin{align}
    \nonumber
    \int 1_{X_k} \, d f_{*}^n ( \lambda |_{E_r} ) 
    &= \sum_{ j = 1 }^n\int_{ \{ \tau^{(k)} \ge j  \}} 1_Y L^{ n-j  } \rho_r  \, d\mu 
    \\ 
    \nonumber & = \sum_{ j = 1 }^{ n-1 }\int_{ \{ \tau^{(k)} \ge j  \}} 1_Y L^{ n-j } \rho_r  \, d\mu +O(\mu(\tau^{(k)}\ge n)) \\
    &= d_{\alpha} \lambda ( E_r ) 
      \sum_{ j = 1 }^{ n - 1 } \frac{ \mu ( \tau^{(k)} \ge j  ) }{ ( n - j )^{1 - \alpha} } 
      + O \Big( \sum_{ j = 1 }^{n-1} \frac{ \eps_{ n - j} \mu ( \tau^{(k)} \ge j  ) }{ ( n - j )^{1 - \alpha}} \Big)
      +O(n^{-\alpha}).
      \label{eq:push-with-Delta}
  \end{align}
  As \( \mu ( \tau^{(k)} \ge j ) \sim c_k j^{ -\alpha } \) we can conclude from Lemma \ref{prop:series-1} and Lemma \ref{prop:series-2} that
  \begin{equation}
    \label{eq:series-bigger-1}
    \sum_{ j = 1}^{n-1} \frac{ \mu ( \tau^{(k)} \ge  j   ) }{ ( n - j )^{1 - \alpha } } \to c_k \frac{ \pi } { \sin  \pi \alpha } 
    \quand
    \sum_{ j = 1 }^{n-1} \frac{ \eps_{ n - j} \mu ( \tau^{(k)} \ge j  ) }{ ( n - j )^{1 - \alpha}} \to 0. 
  \end{equation}
  Combining \eqref{eq:push-with-Delta} and \eqref{eq:series-bigger-1} we obtain
  \[
    \int_{X_k} \, d f_*^n (\lambda|_{E_r}) 
    = \frac{ c_k }{  c_\tau } \lambda ( E_r ) + o ( 1 )
    = \bar p_k \lambda ( E_r ) + o ( 1 ),
  \]
  concluding the result in the case that \( \alpha \in (0,1) \).

  When $\alpha = 1$, we proceed in the same manner as before. Set
  \[
    \eps_{ n } \coloneqq
    \sup_Y \Big|L^n\rho_r - \frac{ \lambda ( E_r ) }{ c_{\tau} \log n }\Big|
  \]
  and note by Corollary \ref{cor:mixing-infinite} that \( \eps_{ n } = o ( 1 / \log n) \)
  Hence, by Lemma \ref{lem:decomposition},
  \begin{align}
    \nonumber
    \int 1_{X_k} \, d f_{*}^n ( \lambda |_{E_r} ) 
    &= \sum_{ j = 1 }^n\int_{ \{ \tau^{(k)} \ge j \}} 1_YL^{ n-j  } \rho_r \, d\mu 
    \\ 
    \nonumber & = \sum_{ j = 1 }^{ n-2 }\int_{ \{ \tau^{(k)} \ge j \}} 1_YL^{ n-j  } \rho_r \, d\mu + O(\mu ( \tau^{(k)} \ge n - 1 ))
    \\
    &= \frac{ \lambda ( E_r ) }{ c_{\tau } }
      \sum_{ j = 1 }^{ n - 2 } \frac{ \mu ( \tau^{(k)} \ge j ) }{ \log ( n - j ) } + O \Big( \sum_{ j = 1 }^{n-2} \frac{ \eps_{ n - j} \mu ( \tau^{(k)} \ge j ) }{ \log ( n - j )} \Big) + O(n^{-1}).
      \label{eq:push-with-Delta-1}
  \end{align}
  Using Lemma \ref{lem:series-1-alpha-1} and Lemma \ref{lem:series-2-alpha-1} we conclude that
  \begin{equation}
    \label{eq:series-1}
    \sum_{ j = 1}^{n-2} \frac{ \mu ( \tau^{(k)} \ge j  ) }{ \log  ( n - j ) } \to c_k 
    \quand
    \sum_{ j = 1 }^{n-2} \frac{ \eps_{ n - j} \mu ( \tau^{(k)} \ge j ) }{ \log ( n - j ) } \to 0. 
  \end{equation}
  Combining \eqref{eq:push-with-Delta-1} and \eqref{eq:series-1} we obtain
  \(
  \int_{X_k} \, d f_*^n (\lambda|_{E_r}) = \bar p_k \lambda ( E_r ) + o ( 1 ),
  \)
  concluding the proof.
\end{proof}

\begin{proof}[Proof of Theorem \ref{thm:mixing}]
  Since $K(X)$ is dense in $L^1(X, m)$, 
  the conclusion of Proposition~\ref{prop:mixing-main} holds for
  general $\lambda\in \cM_1^{\mathrm ac}(X)$.

Fix $k\in\{1,\dots,d\}$. Then
    \[
\Big|\int_X 1_{ X_k } df_{*}^n (\lambda|_{E_r})-
    \int_X 1_{ X_k } df_{*}^n \lambda\Big|
    =\Big|\int_X 1_{ X_k }\circ f^n\, (1_X-1_{E_r})\,d\lambda\Big|
\le \int_X |1_X-1_{E_r}|\,d\lambda
\]
for all $r\ge1$,
so by the dominated convergence theorem,
$\lim_{r\to\infty}\int_X 1_{ X_k } df_{*}^n (\lambda|_{E_r})=
    \int_X 1_{ X_k } df_{*}^n \lambda$ uniformly in $n$.
Also, $\lim_{r\to\infty}\lambda(E_r)=1$ by the dominated convergence theorem. 
Hence, for any $\eps>0$, there exists $r\ge1$ such that
$\big|\int_X 1_{ X_k } df_{*}^n (\lambda|_{E_r})-
    \int_X 1_{ X_k } df_{*}^n \lambda\big|<\eps$ for all $n\ge1$ and
$\bar p_k|\lambda(E_r)-1|<\eps$.
By the first statement of the proof, we obtain that
$\limsup_{n\to\infty}\big|\int_X 1_{X_k}\circ f^n\,d\lambda-\bar p_k\big|<2\eps$.
Since $\eps>0$ is arbitrary, it follows that
$\lim_{n\to\infty}\int_X 1_{X_k}\circ f^n\,d\lambda=\bar p_k$.
In other words, $\lim_{n\to\infty}f_*^n \lambda(X_k)=\bar p_k$.

  Let $\omega_0=\lim_{i\to\infty} f_*^{n_i}\lambda$ be a subsequential limit of $f_*^n\lambda$.
Then $\omega_0(X_k)=\bar p_k$ for each $k$.
Also, $\omega_0=\nu_p$ for some $p\in\cS_0$ by Lemma~\ref{lem:exact}.
Hence $\omega_0=\nu_{\bar p}$.

We have shown that $\nu_{\bar p}$ is the unique subsequential limit of
$f_*^n\lambda$. By compactness of $\cM_1(X)$, it follows that
$\lim_{n\to\infty}f_*^n\lambda=\nu_{\bar p}$.
\end{proof}

\section{Examples}
\label{sec:examples}

In this section, we apply our main results to intermittent maps.
Throughout, we write $|E|=\Leb(E)$ for measurable subsets $E\subset[0,1]$.

\subsection{Thaler maps}
 \label{sec:thaler}

We define \( \cT \) to be the class of interval maps  \( f : [0,1] \to [0,1] \)  studied in \cite{Tha1980,Tha1983} which satisfy the following:
\begin{description}
  \item[{T1\label{itm:T1}}]
  The exists $0=\xi_1<\xi_2<\dots<\xi_d=1$, $d\ge2$, with $f\xi_k=\xi_k$ and
  $f'(\xi_k)=1$ for $k=1,\dots,d$;
  \item[{T2\label{itm:T2}}]
  There exist subintervals \( I_{ 1 }, \ldots, I_d \), $d\ge2$, with $\xi_k\in \Int I_k$, such that \( \bigcup_{ k = 1  }^d I_k = [0, 1]  \) and such that the restriction \( f|_{  I_k } \) extends to a \( C^{ 2 } \) diffeomorphism \( f_k : \bar{I}_k \to [0,1] \);
  \item[{T3\label{itm:T3}}] \( f'(x) > 1 \) for all \( x \not\in \{ \xi_{ 1 }, \ldots, \xi_d\} \) and there exists an \( \eps > 0 \) so that \( f \) is concave (resp.\ convex) on the interval \( (\xi_k - \eps, \xi_k) \cap I_k \) (resp.\ \( (\xi_k, \xi_k + \eps) \cap I_k \));

  \item[{T4\label{itm:T4}}]
  There exist \( \alpha \in ( 0,1 ] \) and \( b_{ 1 }, \ldots, b_d > 0 \) such that for every \( k = 1, \ldots , d \),
  \[
    | f x - x | \sim  b_k | x - \xi_k |^{ 1 + 1/\alpha } \quad \text{as } x \to \xi_k.
  \]
\end{description}

\begin{remark}
  When \( d = 2 \), conditions~\ref{itm:T1}--\ref{itm:T4} reduce to conditions~\ref{itm:tha1}--\ref{itm:tha4} given in the Introduction.
\end{remark}

Thaler maps $f\in\cT$ can be shown to lie in the abstract setting of Section~\ref{sec:abstract} and hence Theorems~\ref{thm:ae},~\ref{thm:dist-conv} and~\ref{thm:mixing} apply to these examples.
The verification of the hypotheses in Section~\ref{sec:abstract}
is essentially contained in~\cite[Proof of Theorem 4.6]{Ser2020} and the references therein. For completeness we recall the main steps of this argument here. Different inducing schemes are needed in the cases \( d = 2 \) and \( d \geq 3 \), and the case \( d = 2 \) will be treated in the more general setting considered in Section~\ref{sec:clm}, so we focus here on the case \( d \geq 3 \).\footnote{There are some typos in~\cite[Equation~(4.7)]{Ser2020} where $x_i$ should be $x_j$ and $f_j$ should be $f_i$.}

By \cite{Tha1980,Tha1983}, \( f \in \cT \) is conservative and exact with a unique (up to scaling) invariant absolutely continuous measure \( \mu \); moreover $\mu([0,1])=\infty$.
Since the branches $f_k:\overline{I}_k\to[0,1]$ are onto, it is immediate (see Section~\ref{sec:gm}) that $f$ is a topologically mixing Markov map. 

Next, we describe the inducing scheme.
Set
\[
  X_k \coloneqq I_k \cap f^{ -1 } I_k,\quad
  k = 1 ,\ldots, d;
  \quand
  Y \coloneqq [0,1]\setminus (X_1\cup\dots\cup X_k).
\]

Let
$Y_k \coloneqq I_k \setminus X_k$. Then $Y=Y_1\dot\cup\cdots \dot\cup Y_d$.
Let \( g_k \coloneqq f_k^{ -1 } : [0,1] \to \overline{I}_k \) denote the inverse branch of \( f \) on \( \overline{I}_k \).
Set
\( X_{ k, n } \coloneqq g_k^n Y_k \) for \( n \ge 1 \), $k=1,\dots, d$. Notice that \( \{X_{ k,n },\;n\ge 1\} \) is a partition of \( X_k \) and that the maps \( f : X_{  k , n + 1 } \to X_{ k, n }  \), \( f:X_{ k, 1 } \to Y_k \) are bijections. 
For \( j \neq k \), set 
\[
  Y_{j,k,1} \coloneqq g_jY_k \quand  Y_{ j, k ,  n } \coloneqq g_j X_{ k, n - 1 },\; n\ge2.
\]
Then $\{Y_{j,k,n}:j\neq k,\,n\ge 1\}$ is a partition of $Y_j$ for each $j$ and hence
\[
  \cP_Y \coloneqq \{ Y_{j,k,n}: j\neq k,\,n\ge 1\}
\]
is a partition of $Y$.

Let $F=f^\tau:Y\to Y$ be the first return map to $Y$.
Then $\tau=n$ on $\bigcup_{j\neq k}Y_{j,k,n}$ and
$F=f^n: Y_{ j , k,  n } \to Y_k $ is a bijection for all $j\neq k$ and all $n$.

\begin{proposition}\label{prop:thaF}
  The first return map \( F : Y \to Y \)
  is a topologically mixing Gibbs-Markov map with respect to the partition $\cP_Y$. 
In addition,~\ref{itm:H1},~\ref{itm:H2}(a,b) and~\ref{itm:H3} hold.
\end{proposition}

\begin{proof}
  By construction, \( F \) is Markov with respect to the partition $\cP_Y$. As \( F Y_{ j,k,n } = Y_k \) it is clear that \( F \) has finite images. Moreover, as \( F Y_k = \bigcup_{ \ell \neq k } Y_{ \ell } \) we see that \( F^{ 3 }  Y_{ j,k,n }  = Y \) and so \( F \) is topologically mixing.
  
  For \( 0 \leq m \leq n - 1 \) and \( x ,y \in [0,1]  \),
  \begin{equation}
    \label{eq:dist-calc-1}
    \log \frac{  (f^{ n - m })' (f^m x)  }{ (f^{ n - m})' (f^m y) } 
    = \sum_{ i = m }^{n-1} \log \frac{ f' ( f^i x) }{ f' ( f^i y) } 
    = \sum_{ i = m}^{n-1} \frac{ |f''( z_i)|  }{  f'( z _i)  } | f^i x -
    f^i y |
    \le |f''|_\infty \sum_{ i = m}^{n-1}  | f^i x - f^i y |,
  \end{equation}
  where \( z_i \in [ f^i x , f^i y ] \) is chosen by the mean value theorem. 
  In particular, for \( x ,y \in Y_{j, k, n} \), 
  \[
    \log \frac{ (f^{ n - m })' (f^m x)  }{ (f^{ n - m})' (f^m y) } 
    \le |f''|_\infty\left( | Y_{ j,k,n } | + \sum_{ i = 1 }^{ \infty } | X_{k,i} | \right) \le |f''|_\infty.
  \]
  Hence for  $i = 0,\ldots, n - 1$,
  \begin{equation}
    \label{eq:dist-calc-2}
    \frac{ | f^i x - f^i  y | }{ | X_{ k , n - i } | }
    \le e^{2|f''|_\infty}\frac{ | f^{n-i} f^ix - f^{n-i}  f^iy | }{ | f^{n-i}X_{k,n-i} | }
    = e^{2|f''|_\infty}\frac{ | F x - F y | }{ | Y_k | },
  \end{equation}
  Inserting~\eqref{eq:dist-calc-2} into~\eqref{eq:dist-calc-1} with \( m = 0 \), we obtain that
  \begin{equation}\label{eq:dist-calc-3}
    \log \Big| \frac{ F' x  }{  F'  y  } \Big|  \leq C | F x - F y |
  \end{equation}
  where $C=|f''|_\infty e^{2|f''|_\infty}\max_k|Y_k|^{-1}$.
  
  Let $\lambda=\inf_Y f'>1$ and set $\theta=\lambda^{-1}$.
  If \( s(x,y) = n  \), then \( 1 \geq | F^n x - F^n y|  \geq    \lambda ^{n - 1} | Fx - Fy | \). Combining this inequality with \eqref{eq:dist-calc-3}, we obtain
  \(
  \log \big| \frac{ F' x  }{  F'  y  } \big|  \leq  C\theta^{n-1} =C\theta^{-1}d_\theta(x,y),
  \)
  which concludes the proof that $F$ is Gibbs-Markov.

We denote the density of $\mu$ by $h$.  
 By~\cite[Lemma~4]{Tha1983}, $h$ is bounded on compact subsets of $[0,1]\setminus\{\xi_1,\dots,\xi_d\}$. Hence, \ref{itm:H1} is satisfied.
Also, it is immediate from the definitions that~\ref{itm:H2}(a,b) hold.

Finally, we  verify~\ref{itm:H3} taking $K(X)=C^1(X)$.
For this, we use very rough estimates (sharp estimates not depending on $r$ are available but not needed). 
Let $\rho \in C^1(X)$ and note that 
\[
Q_r^\rho(y)=\sum \frac{\rho(x)}{(f^r)'(x)}
\quad\text{for $y\in Y$},
\]
where the sum is over those $x\in X$ with $f^rx=y$ such that
$x,fx,\dots,f^{r-1}x\not\in Y$. There are at most $d$ such preimages $x$, one in each $X_k$, so $|Q_r^\rho|_\infty \le d |\rho|_\infty$.
Next, given $y,y'\in Y$ we denote the corresponding preimages $x,x'\in X_k$.
Then
\begin{equation}\label{eq:checking-h3}
|Q_r^\rho(y)-Q_r^\rho(y')|\le\sum |\rho(x)-\rho(x')| + \sum|\rho|_\infty\Big|\frac{1}{(f^r)'(x)}-\frac{1}{(f^r)'(x')}\Big|.
\end{equation}
Certainly,
\begin{align*}
& |(f^r)'(x)-(f^r)'(x')|  =\bigg|\prod_{j=0}^{r-1}f'(f^jx)-
\prod_{j=0}^{r-1}f'(f^jx')\bigg|
 \le |f'|_\infty^{r-1}\sum_{j=0}^{r-1}|f'(f^jx)-f'(f^jx')|
\\ & \quad\; \le |f'|_\infty^{r-1}|f''|_\infty\sum_{j=0}^{r-1}|f^jx-f^jx'|
\le r|f'|_\infty^{r-1}|f''|_\infty|f^rx-f^rx'|
\le r|f'|_\infty^{r-1}|f''|_\infty|y-y'|.
\end{align*}
It follows that
\[
\Big|\frac{1}{(f^r)'(x)}-\frac{1}{(f^r)'(x')}\Big|\le 
|(f^r)'(x)-(f^r)'(x')|
\le r|f'|_\infty^{r-1}|f''|_\infty|y-y'|.
\]
Also, $|\rho(x)-\rho(x')|\le |\rho'|_\infty |x-x'|
\le |\rho'|_\infty |y-y'|$.
Hence, 
\[
|Q_r^\rho(y)-Q_r^\rho(y')|\le d\big(|\rho'|_\infty+r|\rho|_\infty|f'|_\infty^{r-1}|f''|_\infty\big)|y-y'|.
\]

As shown above, 
$|y-y'|\lesssim d_\theta(y,y')$, so
$|Q_r^\rho(y)-Q_r^\rho(y')|\lesssim d_\theta(y,y')$,
yielding $Q_r^\rho\in \cB_\theta(Y)$ as required.
\end{proof}

Let $e_k=1$ for $k=1,d$ and $e_k=2$ for $2\le k\le d-1$. 
Define 
\[
  c_k \coloneqq
  e_k b_k^{ -\alpha } \alpha^\alpha
  \sum_{j = 1, \ldots, d,\; j \neq k}
  h ( g_j \xi_k ) g_j' (  \xi_k )  .
\]
In particular, for $d=2$, we have
\begin{equation}\label{eq:tha-2}
  c_1 = b_1^{ -\alpha } \alpha^\alpha h ( g_2 0 ) g_2' ( 0)  , \quad
  c_2 = b_2^{ -\alpha } \alpha^\alpha h ( g_1 1 ) g_1' ( 1)  .
\end{equation}

The remaining ingredients, namely~\ref{itm:H2}(c)
 and~\ref{itm:H4}, follow from the next result.

\begin{lemma}
  \(
  \mu ( \tau^{ ( k ) } = n )
  \sim
  \alpha c_k
  n^{-(1+\alpha)}
  \)
  as $n\to\infty$ for $k=1,\dots,d$.
\end{lemma}

\begin{proof}
  Note that $\{ \tau^{ (k) } = n \} = \bigcup_{ j = 1, \ldots, d,\; j \neq k   } Y_{ j, k, n - 1 }$.

  Fix $j,k\in\{1,\dots,d\}$, $j\neq k$.
  For \( k \neq 1, d \), the sets \( X_{ k, n } \) have two connected components \( X_{ k , n }^\pm \) which lie to the left or right of \( \xi_k \) and both accumulate at \( \xi_k \) as $n\to\infty$.
  Accordingly, define \( Y_{ j,k,n }^\pm  \coloneqq g_j  X_{ k,n-1 }^\pm  \).
  It suffices to show that $\mu(Y_{j,k,n}^\pm)$ for $k\neq 1,d$ and
  $\mu(Y_{j,k,n})$ for $k=1,d$ have the asymptotic 
  \[
    b_k^{- \alpha } \alpha^{ 1 +  \alpha }
    h ( g_j \xi_k ) g_j' ( \xi_k ) n^{-(1+\alpha)}
    \quad\text{as $n\to\infty$}
    .
  \]
  We give the details for \( \mu ( Y_{ j,k,n }^+ ) \), \( k \neq 1,d \), the other cases being similar.

  By~\cite[Lemma 4]{Tha1983}, $h$ is continuous on  $[ 0 , 1 ] \setminus \{ \xi_{ 1 }, \ldots, \xi_{ d } \}$ (in the standard topology).
  Hence, 
  \begin{equation}\label{eq:h-calc}
    \mu ( Y_{ j , k, n  }^+ ) - h ( g_j\xi_k ) | Y_{ j,k,n }^+ |
    =
    \int_{ Y_{ j , k , n }^{ + } } ( h - h( g_j \xi_k ) )\,d\Leb
    \leq | Y_{ j , k , n }^{ + }  |  \sup_{  Y_{j,k,n}^{ + } } | h - h( g_j \xi_k ) |
  \end{equation}
  It follows that
  $\mu ( Y_{ j , k, n  }^+ )  \sim h ( g_j \xi_k ) | Y_{ j,k,n }^+ |$.
  By the mean value theorem, there exists \( z_n \in X_{ k,n }^+ \) so that
  $| Y_{ j,k,n }^+ | = g_j' (z_n) | X_{ k,n-1 }^+ | \sim g_j'( \xi_k ) | X_{ k,n-1 }^+|$. Combining these last two estimates,
  \begin{equation}
    \label{eq:mvt}
    \mu ( Y_{ j , k, n  }^+ )  \sim h ( g_j \xi_k ) 
    g_j'( \xi_k ) | X_{ k,n-1 }^+|. 
  \end{equation}
  It remains to estimate \( | X_{k,n }^+ | \). We recall the following standard calculation.

  \begin{sublemma}[{see for example \cite[Lemma 4.8.6]{Aar1997}}]
    \label{sublem:recursion}
    Suppose that \( T : [0, C] \to [0, \infty) \) is such that \( Tx > x \) for each \( x \in (0,C] \) and  \( Tx \sim x+ b x^{ 1 + p } \) as \( x \to 0 \). Let \( z_{ n  } = T  z_{ n + 1  }   \). Then
    \[
      z_n \sim ( p b n )^{ -1/p }
      \quand
      z_{ n  } - z_{  n + 1 }  \sim b^{ -1/p } ( p n  )^{ -(1+ 1/p) }
      \quad \text{as } n \to \infty
    \]
  \end{sublemma}

  \begin{proof}
    It follows from the assumptions that the sequence \( z_n \) is strictly decreasing with \( z_n \to 0 \). We then calculate
    \[
      z_n^{-p}\sim
      z_{ n+1 }^{- p }  (1 + b z_{ n + 1 }^{ p } )^{ -p }
      = z_{n+1}^{-p}  ( 1 - p  b z_{n+1}^{p} + o ( z_{n+1}^{2p} ))
      = z_{n+1}^{- p }  - pb + o(1)
    \]
    as \( n\to \infty \).
    Summing over $n$, this yields  \(z_{n}^{-p} = p b n + o(n)\) which gives the estimate for $z_n$. Finally, \(  z_n - z_{ n + 1 }  =  Tz_{ n+1 } - z_{ n + 1 }  \sim b z_{ n + 1 }^{ 1+p} \sim b^{ -1/p } ( p n  )^{ -(1+ 1/p) } \).
  \end{proof}

  Recalling~\ref{itm:T4}, the estimate for $| X_{ k, n }^{+} |$  reduces after a  change of coordinates to the situation in
  Sublemma~\ref{sublem:recursion} with \( p = 1/\alpha \) and \( b = b_k \). Hence
  \( | X_{ k, n }^{+} | \sim b_k^{ -\alpha } (n/\alpha)^{ -(1 + \alpha) } \). 
  Combining this with \eqref{eq:mvt}, we obtain
  \[
    \mu ( Y_{ j , k, n }^+ ) 
    \sim  h ( g_j \xi_k ) g_j' ( \xi_k ) b_k^{- \alpha } (\alpha/n)^{ 1 +  \alpha }
  \]
  as required.
\end{proof}

\subsection{Intermittent maps with critical points and/or singularities}\label{sec:clm}

We now consider a class of intermittent interval maps with two branches that possibly admit a critical point and/or a singularity at the discontinuity. This class of maps include the maps in \( \cT \) with \( d = 2 \) and the maps described in \cite{CoaLuzMuh2023}. We define \( \cF \) to be set of maps \( f : [-1,1] \to [-1,1] \) which satisfy the following conditions.

\begin{description}
  \item[{F0\label{itm:F0}}]
  There exists a \( c \in (-1,1) \) such that the restrictions
  \(f_{-} \coloneqq f : (-1,c) \to (-1,1) \) and \(f_{+} : (c,1) \to (-1,1)\) are \(C^2\) orientation preserving diffeomorphisms with no fixed points.
\end{description}

\begin{remark}
  For notational simplicity we will assume that \( c = 0 \). Notice also that we do not assume that the \(f_{ \pm }\) extend to \(C^2\) functions on the closure of \(I_{ \pm }\).
\end{remark}

\begin{description}
  \item[{F1\label{itm:F1}}] There exist \(\ell_{+},\ell_{-} > 0\) and \(k_+,k_- > 0\) such that
  \begin{equation}
    \label{eq:def-clm}
    f (x)
    =
    \begin{cases}
      x + b_-(1 +  x )^{ 1 + \ell_{-} } + o ( (1 + x)^{ 1 + \ell_- } ), &\text{for } x \in U_{ -1 }\\
      1 - a_- | x |^{ k_- },  &\text{for } x \in U_{ 0- }\\
      - 1 + a_+ | x |^{ k_+ }, &\text{for } x \in U_{ 0+ }\\
      x - b_+( 1 -  x )^{ 1 + \ell_{+} } + o ( (1 - x)^{ 1 + \ell_+ } ) & \text{for } x \in U_{ +1 }\\
    \end{cases}
  \end{equation}
  whenever \( k_+, k_- \neq 1 \)
  for some \(a_{ \pm }, b_{ \pm } > 0\), and some neighbourhoods \( U_{ -1 }, U_{ 0- } \) of \( -1, 0 \) in \( [-1,0] \) and some neighbourhoods \( U_{ +1 }, U_{ 0+ } \) of \( 1,0 \) in \( [0,1] \).
  If \( k_+ = 1 \) and/or \( k_- = 1 \), then we replace the corresponding lines in~\eqref{eq:def-clm} with the assumption that \( f' (0  - ) = a_- > 1 \) and/or  \(f' (0+) = a_+ > 1 \) respectively.
\end{description}

\begin{remark}\label{rem:image-of-U-0}
  It is convenient to assume that \( f  U_{ 0 \pm }  \subset U_{ \mp 1 } \). Notice that this assumption poses no restriction on~\ref{itm:F1} as \( U_{ 0 \pm } \) can be taken to be arbitrarily small.
\end{remark}

\begin{remark}
  This definition is more general than the one in~\cite{CoaLuzMuh2023} as~\ref{itm:F1} stipulates only an asymptotic behaviour near the fixed points. However, it does not include the maps in~\cite{MubSch2022} due to the restriction mentioned in Remark~\ref{rmk:H2}. We expect that our results hold also for the maps in~\cite{MubSch2022}.
\end{remark}

Suppose that \(f\) satisfies~\ref{itm:F0} and let \(\gamma_{ \pm } \in I_{ \pm }\) be the two points of period \( 2 \) for \(f\). Define the intervals
\begin{equation}\label{eq:clm-def-partition}
  Y \coloneqq [\gamma_- , \gamma_+ ], \quad
  X_{\pm, n} \coloneqq f^{ -n }_{ \pm } Y, \quand
  Y_{ \pm, n + 1 } \coloneqq f_{ \pm }^{ - 1} X_{ \mp, n }.
\end{equation}
By definition, for $n \ge 1$,  \( f : Y_{ \pm, n + 1 } \to X_{ \mp, n } \) and \( f : X_{ \pm, n } \to X_{ \pm, n - 1 } \) and $f : X_{\pm, 1} \to Y$ are bijections. Moreover, from~\ref{itm:F0}, we know that $X_{\pm,n}$ are consecutive intervals, accumulating at $\pm 1$. So, the $\{ X_{\pm, n } : n \ge 1\}$ partition $X \setminus Y$ and, as $f (Y) = X \setminus Y$, the $\{ Y_{\pm, n} : n \ge 2\}$ partition $Y$. Thus, \(\cP \coloneqq \{ Y_{ \pm, j }, X_{ \pm , n } : n \geq 1, j \geq 2 \}\) forms a Markov partition for $f$. Our final condition will ensure that our maps have good expansion and distortion properties.

\begin{description}
  \item[F2\label{itm:F2}]
  \( f \) is convex (resp.\ concave) on \( U_{ -1 } \) (resp.\ \( U_{ 1 } \)) and moreover
  \begin{enumerate}
    \item\label{itm:F2.a}
    If \( f \) is not \( C^2 \) on \( \overline U_{-1} \) (resp.\ \( \overline U_{+1} \)), then \( f '' (x) \lesssim (1 + x)^{ \ell_- - 1 }  \) (resp.\ \( |f''( x )| \lesssim ( 1 - x )^{ \ell_+ - 1 }  \)).

    \item\label{itm:F2.b}
    If \( k_{ \pm } \neq 1 \), then there exists a \( \lambda > 1 \) such that  \( ( f^n )' (x) > \lambda  \)  for each \( x \in Y_{ \pm , n} \) and each \( Y_{ \pm , n} \not\subset U_{ 0\pm } \).
  \end{enumerate}
\end{description}

\begin{remark}
  Notice that~\ref{itm:F2}.\ref{itm:F2.a} is only assumed when the map is not \( C^{ 2 } \) at the fixed points and that~\ref{itm:F2}.\ref{itm:F2.b}  is only an assumption about the map outside of the neighbourhoods \( U_{ 0\pm }, U_{ \pm 1 } \) and is trivially satisfied if \( f ' (x) > 1 \) for each \( x \not\in \{ -1, 1 \} \).
\end{remark}

We let \( \alpha_+ \coloneqq 1 / \ell_+k_- \), \( \alpha_- \coloneqq 1 / \ell_- k_+ \) and define
\[
  c_{ 1 } \coloneqq
  h(0) a_-^{ - 1/k_-} ( \ell_+ b_+ )^{ - 1 / \alpha_- } \alpha_-^{2},
  \quand
  c_{ 2 } \coloneqq
  h(0) a_+^{ - 1/k_+} ( \ell_- b_- )^{ - 1 / \alpha_+ } \alpha_+^{2}.
\]
We will assume that \( \alpha_{ + } = \alpha_{ - } = \alpha \in (0,1) \).

\begin{theorem}\label{thm:clm}
  Suppose that \( f \in \cF \) with \(  \alpha \in (0,1)\). Then~\ref{itm:H1}--\ref{itm:H4} hold.
\end{theorem}

Throughout this section we fix
\[
  X_-\coloneqq [0, \gamma_-), \quad
  X_+ \coloneqq (\gamma_+, 0],
\]
and let \( \tau : Y \to \N \) denote the first return time to \( Y \) and \( \tau^{ (\pm) } (x) \coloneqq \Card \{ n \leq \tau (x) : f^n (x) \in X_{ \pm } \} \). Let \( F:Y\to Y \) be the first return map. By construction, we obtain the following lemma.

\begin{lemma}\label{lem:clm-first-consequence}
  \( F|_{ Y_{ \pm , n} } = f^n \) and the interval \(Y\) dynamically separates \(X_-, X_+\).
\end{lemma}

In the remainder of this section we will show that \( f \) satisfies assumptions~\ref{itm:H1}--\ref{itm:H3} and thus conclude the proof of Theorem~\ref{thm:clm}. Notice that \( X_{ \pm, n } \) and \( Y_{ \pm , n} \) here play the same roles as \( \Delta_n^{ \pm  } \)  and \( \delta_n^{ \pm } \) respectively in \cite{CoaLuzMuh2023}. The next Lemma shows that the asymptotic behaviour of the sizes of these partition elements remains exactly the same as in~\cite{CoaLuzMuh2023}.

\begin{lemma}
  \label{lem:clm-partition-estimates}
  \[
    | X_{ \pm, n } | \sim b_{ \pm }^{ - 1 / \ell_{ \pm }} ( \ell_{ \pm }  n )^{ -1 - 1 / \ell_{ \pm } }
    \quand
    \Leb ( \tau^{(\pm)} = n )
    = | Y_{ \pm , n} |
    \sim a_{ \pm }^{ - 1/k_{ \pm }} ( \ell_{ \mp } b_{ \mp } )^{ - \alpha_{ \pm } } \alpha_{ \pm } n^{-1 -  \alpha_{ \pm }}.
  \]
\end{lemma}

\begin{proof}
  We will only explicitly prove the estimates for \( Y_{ + , n} \) and \( X_{ - , n } \) as the other estimates follow in the same way.
  Let \(\gamma_n \coloneqq f^{-n}_- \gamma\) so that  \( X_{ - , n } = [\gamma_{n}, \gamma_{n-1}]\).
  As in the previous section we can apply Sublemma~\ref{sublem:recursion} to the sequence \( z_{n} \coloneqq 1 + \gamma_{n} \) with \( b = b_- \) and \( p = \ell_- \) to obtain
  \begin{equation}
    \label{eq:clm-bound-1-gamma}
    z_n \sim (b_- \ell_- n )^{ - 1 / \ell_-},
    \quand 
    | X_{ \pm, n } | \sim b_{ \pm }^{ - 1 / \ell_{ \pm }} ( \ell_{ \pm }  n )^{ -1 - 1 / \ell_{ \pm } }.
  \end{equation}

  Now, notice that as \(y \downarrow 0\) we have by definition \( f_{+}^{-1} y = \big(\frac{ y + 1 }{a_+} \big)^{1/k_+} \). It follows that \(f_+^{-1}  \gamma_n  = \big(\frac{1 + \gamma_n}{a_+}\big)^{1/k_+}\). 
From \eqref{eq:clm-bound-1-gamma} we obtain  \(f^{ -1 }_+\gamma_n \sim a_+^{ - 1/k_+} ( \ell_- b_- n)^{-1/ (\ell_- k_+)} \) which yields the claimed asymptotics for \( |Y_{ \pm , n} | \).
\end{proof}

\begin{proposition}\label{prop:clm-gm}
  The map \( F : Y \to Y \)
  is a topologically mixing Gibbs-Markov map with respect to the partition $\cP_Y$.
\end{proposition}

Having established Lemma~\ref{lem:clm-partition-estimates} the proof of Proposition~\ref{prop:clm-gm} follows essentially verbatim from~\cite{CoaLuzMuh2023} replacing the roles of \( \delta_n^{ \pm } \) with \( Y_{ \pm , n} \) and \( \Delta_n^{ \pm } \) with \( Y_n^{ \pm } \). For completeness we include the main steps of this argument.

\begin{lemma}\label{lem:clm-expansion}
  There exists \( \lambda > 1 \) such that \( F'(y) > \lambda \) for all \( y \in Y \).
\end{lemma}

\begin{proof}
  We follow the proof of~\cite[Proposition 3.6.]{CoaLuzMuh2023}.
  We only consider the case \( y \in Y_{ +, n} \) as the case that \( y \in Y_{ - , n} \) is the same. Define the function \( \phi \coloneqq f_+^{-1} \circ f_- \circ f_+ \) and notice that \( \phi : Y_{ + , n + 1} \to Y_{ + , n } \) bijectively.

  \begin{sublemma}
    Let \( y \in Y_{ + , n} \) where \( n \) is such that \( Y_{ + , n} \subset U_{ 0+ } \), then
    \(
      \dfrac{ f' (y) f' (f y) }{ f' ( \phi (y)) }  > 1.
    \)
  \end{sublemma}

  \begin{proof}
    If \( k_+ \in (0,1) \) then \( f' \) is decreasing on \( U_{0+} \) and so, as \( y < \phi(y) \), we have that \( f ' ( y ) / f ' ( \phi (y ) ) > 1 \). By construction \( f(y) \in U_{-1} \) for every \( y \in U_{0+} \), so \( f' ( f y ) > 1 \) and we are finished.

    Assume now that \( k_+ > 1 \) and to ease notation set \( k = k_+, a = a_+, b = b_-, \ell = \ell_- \). Setting \( x = fy \) we recall that \( x \in U_{ -1} \) by Remark~\ref{rem:image-of-U-0} and using the convexity of \( f \) on \( U_{ -1 } \) we obtain that 
    \(
      f  x   \leq f( -1 ) + f ' ( x ) ( x + 1 )
      = -1 + f'(x)  ay^k\)
which in turn implies that
      $\phi (y) 
      \leq f' (x)^{ 1/k } y$.
    As \( y< \phi(y)\) and as \( ( f '(y) )/ ( f' ( \phi( y)) ) = ( y / \phi (y) )^{ k - 1 } \),
    \[
      \frac{f ' (y)}{ f' ( \phi (y)) } f' (fy)
      \geq \Big( \frac{\phi(y)}{y} \Big) \Big( \frac{ f' (x)^{ 1/k } y }{ y } \Big)^{ -k } f'(x)
      \geq 1.
    \]
  \end{proof}

  Let \( m_+ = \min \{ m : Y_{ +,  m } \subset U_{ 0+ } \} \). Condition~\ref{itm:F2}
  implies that \( F ' y \geq \lambda \)  for all  \( y \in Y_{ +, m }\) whenever \( m \leq m_+ \).
  The Sublemma above allows us to conclude that if $y \in Y_{+, m + 1}$ and $m + 1 \ge m_+$ then
\[
    (f^{m+1}) ' (y) = f'(y) \cdot f' (f y) \cdots f' (f^{m} y)
    = \frac{ f'(y) \cdot f' (f y) }{ f' ( \phi (y)) }  (f^{m})' ( \phi (y))
    \geq (f^{m})' ( \phi (y)),
\]
  and so \( F' y \ge F ( \phi (y)) \). Proceeding inductively, for any $y \in Y_n \subset U_{0+}$, we obtain \( F ' (y) \geq F ' ( \phi^{ n + 1 - m^+ } (y)) \ge \lambda\).
\end{proof}

\begin{lemma}\label{prop:clm-distortion}
  There exists a \( C > 0 \) and a \( \theta \in (0,1) \) such that
  \[
    \left| \log \frac{ F ' (x) }{ F'(y) } \right| \leq C \theta^{ s ( x, y ) }.
  \]
\end{lemma}

\begin{proof}
  We only consider \( x,y \in Y_{ + , n} \) as the argument for \( x,y\in Y_{ -,  n } \) is the same. Calculating as in~\eqref{eq:dist-calc-1} one finds 
  \begin{align}
    \left| \log \frac{ ( f^{ n - j } )' ( f^{j} x) }{ ( f^{ n - j } ) ' ( f^{j} y) } \right|
    \leq
    \label{eq:clm-dist-bound-2}
    \left| \frac{ f''  u_0  }{ f' u_0 } \right| | Y_{ + , n} |
    + \sum_{k=1}^{ n - 1} \left| \frac{ f''  u_k }{ f'  u_k } \right| | X_{ - , n - k } |
  \end{align}
  Notice the \( f'' ( u_{ 0 } ) / f' ( u_{ 0 } ) \lesssim u_{ 0 }^{ -1 } \lesssim  (f^{ -1 } \gamma_n)^{ -1 } \lesssim ( 1 + \gamma_n )^{ -1/k_+ } \) and~\eqref{eq:clm-bound-1-gamma}  gives that \( ( 1 + \gamma^-_n )^{ 1 / k_+  } \lesssim n^{ -\alpha_- } \). So, as \( |Y_{ + , n}| \lesssim n^{ -1 - \alpha_- } \) 
  we find that the first term in~\eqref{eq:clm-dist-bound-2} is uniformly bounded in \( n \). If \( f \) is \( C^{ 2 } \) on \( \overline U_{ -1 } \) the second term in \eqref{eq:clm-dist-bound-2} is summable and whence uniformly bounded. Otherwise, by assumption \ref{itm:F2},~\eqref{eq:clm-bound-1-gamma} and Lemma \ref{lem:clm-partition-estimates} we have that \( f''( u_k ) | Y_{ + , n } | \lesssim ( 1 + \gamma_{ n - k }^- )^{ \ell_- - 1 } |Y_{ + , n - k }| \lesssim ( n - k )^{ -1 + 1/\ell_- } ( n - k )^{ -1 - 1/\ell_{ -} } \lesssim ( n - k )^{ -2 } \) which is summable.
  Thus,~\eqref{eq:clm-dist-bound-2} is uniformly bounded  and we can calculate as in~\eqref{eq:dist-calc-2} to find \( \log | F'x / F'y | \leq C | F x  - F y | \). Thus, proceeding in the same way as in the proof of Proposition~\ref{prop:thaF} one concludes the result with \( \theta \coloneqq \tilde{ \lambda }^{ - 1 } \).
\end{proof}

\begin{proof}[Proof of Proposition~\ref{prop:clm-gm}]
  The fact that $F$ is Gibbs-Markov with respect to $\cP_Y$ follows from Lemmas~\ref{lem:clm-expansion} and~\ref{prop:clm-distortion}.
  From~\eqref{eq:clm-def-partition} we obtain $F (Y_{ \pm , n } ) = f^{ n } ( Y_{ \pm , n } ) = Y$, so $F$ is full-branched and hence topologically mixing.
\end{proof}

\begin{lemma}\label{lem:clm:h3}
  $f$ satisfies~\ref{itm:H3}.
\end{lemma}

\begin{proof}
  We follow the argument given in Proposition~\ref{prop:thaF}. The main difference between our current setting and that of Proposition~\ref{prop:thaF} is that $|f'|$ is no longer necessarily bounded away from $0$ and $|f''|$ may not be finite at the points $\pm 1, 0\pm$ so we have to take a little more care.

  Let $K(X)=C^1(X)$,
  and let $\rho \in C^1(X)$. Then 
  $Q_r^\rho(y)= \frac{\rho(x_+)}{(f^r)'(x_+)} + \frac{\rho(x_-)}{(f^r)'(x_-)}$
for $y\in Y$,
where $x_{\pm} \in X_{\pm,r}$ is such that $f^r x_\pm = y$. Thus, $|Q_r^\rho|_\infty \le 2 |\rho|_\infty$.

  Let $y, y' \in Y$ and let $x,x' \in X_{\pm}$ be such that $f^r x = y$ and $f^r x' = y'$. Set $X^{(r)} = \bigcup_{ k = 1 }^{ r } X_{-, k }\cup X_{+,k}$. Notice that~\eqref{eq:checking-h3} holds as before, but now
  \begin{align*}
    |(f^r)'(x)-(f^r)'(x')| 
    &\le \left| f'|_{X^{(r)}} \right|_\infty^{r-1}
    \sum_{j=0}^{r-1} \left| f'(f^jx)-f'(f^jx') \right| \\
    &\le \left| f'|_{X^{(r)}} \right|_\infty^{r-1} \left|f''|_{X^{(r)}} \right|_\infty\sum_{j=0}^{r-1} |f^jx-f^jx'|.
  \end{align*}
  Moreover, $\inf_{ j = 0,\ldots, r - 1, \zeta \in X_{ - ,  r - j }}  (f^{ r - j})' (\zeta) > 0$ and so $|f^{j} x - f^j x'| \lesssim |f^r x - f^r x'|$,
  where the implied constant depends on $r$.  Thus,
  \[
    \Big|\frac{1}{(f^r)'(x)}-\frac{1}{(f^r)'(x')}\Big|
    \lesssim  |(f^r)'(x)-(f^r)'(x')|
    \lesssim | y - y' |.
  \]
  As before, $|\rho(x)-\rho(x')|\le |\rho'|_\infty |x-x'| \le |\rho'|_\infty |y-y'|$.
  Hence, by~\eqref{eq:checking-h3}
  \[
    |Q_r^\rho(y)-Q_r^\rho(y')| \lesssim |y-y'|.
  \]
  Finally, $F' (x) \ge \lambda$, so $|y - y'|\lesssim d_{\theta} (y,y')$ where $\theta$ is as in Lemma~\ref{prop:clm-distortion}. Hence $Q_r^\rho \in \cB_{\theta} (Y)$ as required.
\end{proof}

\begin{proof}
  [Proof of Theorem~\ref{thm:clm}]
  Proposition~\ref{prop:clm-gm} yields~\ref{itm:H1} and Lemma~\ref{lem:clm:h3} yields~\ref{itm:H3}. So, it only remains to show~\ref{itm:H2} and~\ref{itm:H4}. By Lemma~\ref{lem:gm} \( F \) preserves an absolutely continuous invariant measure \( \mu \) with density \( h \). We claim that \( h \) is continuous on \( Y \). Given the validity of this claim we can apply the same argument as in~\eqref{eq:h-calc} to conclude from Lemma~\ref{lem:clm-partition-estimates} that
  \[
      \mu ( \tau^{ \pm } = n ) = \mu( Y_{ \pm, n } )
      \sim h(0) | Y_{ \pm,  n } |
      \sim  h(0) a_{ \pm }^{ - 1/k_{ \pm }} ( \ell_{ \mp } b_{ \mp } )^{ - \alpha_{ \pm } } \alpha_{ \pm } n^{-1 -  \alpha_{ \pm }}.
  \]
  Together with Lemma~\ref{lem:clm-first-consequence}, this concludes~\ref{itm:H2} and~\ref{itm:H4}. 

  Let \( \cP _{ Y }^{ ( n ) } \) denote the \( n^{ th } \) refined partition of \( F \) and let \( a \in \cP_{ Y }^{ (n) } \).
  From the uniform expansion of \( F \) and the mean-value theorem we know that for \( j \leq n \), \( | F^{ j } x - F^{ j }y | \leq \lambda^{ -( n - j ) } | F^{ n } x - F^{ n } y | \) for every \( x,y \in a \). Moreover,  recalling that \( \log | F' x / F' y | \leq C | F x - F y | \) for \( x, y \in Y \) it follows that for any \( x,y \in a \)
  \begin{equation}\label{eq:dist-calc-again}
    \log \left| \frac{ ( F^{ n } )' x  }{ ( F^{ n } )' y } \right|
    \leq \sum_{ j = 0 }^{ n - 1 } \left| \frac{ ( F^{ j } )' x  }{ ( F^{ j } )' y } \right|
    \leq \frac{ C }{ \lambda - 1 } | F^{ n } x - F^{ n } y |,
  \end{equation}
  and so
  \( | ( F^{ n }) ' x / ( F^{ n } )'  y ) - 1 | \lesssim | F^{ n } x - F^{ n } y |  \). 
Moreover,~\eqref{eq:dist-calc-again} implies \( 1 / [ ( F^{ n } )' y ] \lesssim | a | \) for all \( y \in a \).
  Using the standard pointwise formula for the transfer operator \( L_{ Y } \)  for \( F \), we obtain
  \begin{align*}
    | L_{ Y }^{ n } 1 (x) - L_{ Y }^{ n } 1 (y) |
    &= \bigg| \sum_{  \tilde x \in F^{ -n } x, \, \tilde y \in F^{ -n  } y}
      \frac{ 1 } { (F^{ n }) ' \tilde x } - \frac{ 1 } { (F^{ n }) ' \tilde y } 
      \bigg|\\
    &\leq \sum_{ \tilde x \in F^{ -n } x, \, \tilde y \in F^{ -n  }  }
      \left| \frac{ (F^{ n }) ' \tilde y }{  ( F^{ n }) ' \tilde x } - 1 \right| \left| \frac{1}{  ( F^{ n } ) ' \tilde y } \right|\\
    &\lesssim  \sum_{ a \in \cP_{ Y }^{ ( n ) } } | x - y | | a |
      \lesssim | x - y |.
  \end{align*}
  Similarly, one can check that \( L_{ Y }^{ n } \) is uniformly bounded. The Arzel\`{a}-Ascoli Theorem then yields that \( \frac{ 1 }{ n } \sum_{ j = 1 }^{ n - 1 }  L_{ Y }^{ j } 1 \) has a subsequence which is uniformly convergent to some Lipschitz \( \tilde h  : Y \to \R \). As \( \tilde h \) is necessarily the density of an invariant measure it follows that \( h 1_{ Y }  = \tilde h \) yielding that \( h \) is Lipschitz (and hence continuous). This establishes our claim and concludes the proof.
\end{proof}

\appendix
\section{Series estimates}

In this appendix, we carry out some calculations required in Section~\ref{sec:mixing}.

\begin{lemma} \label{prop:series-1}
  Let $\alpha\in(0,1)$. Then
  $\lim_{n\to\infty}\sum_{j=1}^{n-1}(n-j)^{\alpha-1}j^{-\alpha}=\tfrac{\pi}{\sin\alpha\pi}$.
\end{lemma}

\begin{proof}
  The function $(n-x)^{\alpha-1} x^{-\alpha}$ has one critical point at 
  $x=\alpha n$. Hence, by approximating integrals by Riemann sums on the intervals
  $[0,\alpha n]$ and $[\alpha n,n]$,
  \[
    \sum_{j=1}^{n-1}(n-j)^{\alpha-1}j^{-\alpha}
    =
    \int_0^n (n-x)^{\alpha-1}x^{-\alpha}\,dx + O(n^{-\alpha})+O(n^{\alpha-1}).
  \]
  Using standard properties of Beta and Gamma functions, we obtain
  \begin{align*}
    \int_0^n (n-x)^{\alpha-1}x^{-\alpha}\,dx 
    & = \int_0^1 (1-x)^{\alpha-1}x^{-\alpha}\,dx 
      =\operatorname{B}(\alpha,1-\alpha)
      =\Gamma(\alpha)\Gamma(1-\alpha)/\Gamma(1)
    \\ & =\Gamma(\alpha)\Gamma(1-\alpha)
         =\tfrac{\pi}{\sin\alpha\pi},
  \end{align*}
  as required.
\end{proof}

\begin{lemma} \label{prop:series-2}
  Let $\alpha\in(0,1)$. Suppose that $g:\N\to\R$ satisfies $\lim_{n\to\infty}g(n)=0$.
  Then
  $$\lim_{n\to\infty}\sum_{j=1}^{n-1}j^{-\alpha}g(j)\,(n-j)^{\alpha-1}=0$$ 
  and
  $\lim_{n\to\infty}\sum_{j=1}^{n-1}(n-j)^{\alpha-1}g(n-j)\,j^{-\alpha}=0$.
\end{lemma}

\begin{proof}
  Set $C=\max_n |g(n)|$.
  Let $\eps\in(0,\frac12)$.
  Then
  \begin{align*}
    \sum_{j=1}^{n-1}j^{-\alpha}g(j)\,(n-j)^{\alpha-1}
    &=
      \sum_{1\le j<\eps n}j^{-\alpha}g(j)(n-j)^{\alpha-1}
      +\sum_{\eps n\le j< n}j^{-\alpha}g(j)(n-j)^{\alpha-1}\\ 
    & \le
      C2^{1-\alpha} n^{\alpha-1}\sum_{1\le j<\eps n}j^{-\alpha}
      +(\eps n)^{-\alpha}\max_{k\ge \eps n}g(k)\sum_{\eps n\le j< n}(n-j)^{\alpha-1} \\ 
    & \lesssim
      n^{\alpha-1}(1+(\eps n)^{1-\alpha})
      +(\eps n)^{-\alpha}\max_{k\ge \eps n}g(k) n^{\alpha} \\ 
    & \lesssim
      n^{\alpha-1}+\eps^{1-\alpha}
      +\eps^{-\alpha}\max_{k\ge \eps n}g(k).
  \end{align*}
  Hence,
  \[
    \limsup_{n\to\infty}\sum_{j=1}^{n-1}j^{-\alpha}g(j)(n-j)^{\alpha-1}
    \lesssim \eps^{1-\alpha}.
  \]
  The first limit follows since $\eps$ is arbitrarily small.
  Replacing $\alpha$ by $1-\alpha$, we obtain the second limit.
\end{proof}

\begin{lemma}
  \label{lem:series-1-alpha-1}
  Suppose that $g(x)=o(1)$ as $x\to\infty$.
  Then $\sum_{j=1}^n \frac{g(j)}{j} =o(\log n)$.
\end{lemma}

\begin{proof}
  For $\eps>0$, choose $n_0\ge 1$ such that $|g(x)|<\eps$ for $x>n_0$.
  Then 
  \[
    \sum_{j=1}^n \frac{g(j)}{j}\le 
    \sum_{j=1}^{n_0} \frac{g(j)}{j}+
    \eps \sum_{j=n_0}^n \frac{1}{j}
    \le \sum_{j=1}^{n_0} \frac{g(j)}{j}+
    \eps \log n.
  \]
  Hence
  $\limsup_{n\to\infty}(\log n)^{-1}\sum_{j=1}^n \frac{g(j)}{j}\le\eps$.
\end{proof}

\begin{lemma}\label{lem:series-2-alpha-1}
  Let $g_1,g_2:[0,\infty)\to\R$ with $\lim_{x\to\infty} g_i(x)=1$ for $i=1,2$.
  Then
  \[
    \lim_{n\to\infty}\sum_{j=1}^{n-2} \frac{g_1(j)}{j}\frac{g_2(n-j)}{\log(n-j)}=1.
  \]
\end{lemma}

\begin{proof}
  Write 
  \[
    \sum_{j=1}^{n-2} \frac{g_1(j)}{j}\frac{g_2(n-j)}{\log(n-j)}=S_1+S_2
  \]
  where
  \[
    S_1=\sum_{1\le j\le n/2} \frac{g_1(j)}{j}\frac{g_2(n-j)}{\log(n-j)}, \qquad
    S_2=\sum_{n/2<j\le n-2} \frac{g_1(j)}{j}\frac{g_2(n-j)}{\log(n-j)}.
  \]

  Now, 
\begin{align*}
|S_2| & \le 
|g_1|_\infty|g_2|_\infty\frac{2}{n} \sum_{n/2\le j\le n-2}\frac{1}{\log (n-j)}
\\ & =|g_1|_\infty|g_2|_\infty\frac{2}{n} \sum_{2\le j\le n/2}\frac{1}{\log j}\le 
|g_1|_\infty|g_2|_\infty\frac{2}{n} \sum_{2\le j\le n}\frac{1}{\log j}. 
\end{align*}
  But 
  \[
    \sum_{2\le j\le n}\frac{1}{\log j}\lesssim \frac{n}{\log n}
    +\sum_{n/\log n\le j\le n}\frac{1}{\log j}
    \le \frac{n}{\log n}+\frac{1}{\log(n/\log n)}n\lesssim \frac{n}{\log n},
  \]
  so $|S_2|\lesssim \frac{1}{\log n}$.

  By Lemma~\ref{lem:series-1-alpha-1}, $\sum_{j=1}^n \frac{g_1(j)}{j}\sim\log n$. Hence,
  \[
    S_1\le \sup_{k\ge n/2}g_2(k)\frac{1}{\log n/2}\sum_{1\le j\le n/2}\frac{g_1(j)}{j}\sim \frac{1}{\log n/2}\log n/2 = 1.
  \]
  Similarly,
  \[
    S_1\ge \inf_{k\ge n/2}g_2(k)\frac{1}{\log n}\sum_{1\le j\le n/2}\frac{g_1(j)}{j}\sim \frac{1}{\log n}\log n/2 \sim 1.
  \]
  Hence $S_1+S_2\sim S_1\sim 1$ as required.
\end{proof}

\subsubsection*{Acknowledgment} We are very grateful to Stefano Luzzatto for important discussions that led to this paper and for the hospitality of ICTP where part of this research was carried out.

We are also greatly indebted to the referees for carefully checking the arguments and for numerous suggestions that improved the clarity of the paper.

\printbibliography

\end{document}